\documentclass[12pt]{article}

\usepackage[english]{babel}
\usepackage[cp1251]{inputenc}
\usepackage[T2A]{fontenc}

\usepackage{amssymb}
\usepackage{amsmath}
\usepackage{amsthm} 
\usepackage{amsfonts}
\usepackage{graphicx,geometry}
\usepackage[title]{appendix}
\usepackage{cite}
\usepackage{float}
\usepackage{xcolor}
\usepackage[font=small,labelfont=bf]{caption}

\usepackage{tikz}
\usetikzlibrary{arrows,decorations.pathmorphing,backgrounds,positioning,fit,petri,patterns}

\renewcommand{\leq}{\leqslant}
\renewcommand{\geq}{\geqslant}

\DeclareFontEncoding{LS1}{}{}
\DeclareFontSubstitution{LS1}{stix2}{m}{n}

\DeclareMathOperator{\sign}{sgn}
\DeclareMathOperator{\supp}{supp}

\newtheorem{theorem}{Theorem}
\newtheorem{lemma}{Lemma}
\newtheorem{proposition}{Proposition}
\newtheorem{corollary}{Corollary}

\theoremstyle{definition}
\newtheorem{definition}{Definition}
\newtheorem{remark}{Remark}

\newenvironment{enbibliography}{\vspace{-0.5cm}\begin{thebibliography}}{\end{thebibliography}}

\usepackage[title]{appendix}

\begin{document} 
\title{Kru\v{z}kov-type uniqueness theorem for a non-monotone flow function case with application to Riemann problem solutions}
\author{Yulia Petrova\footnote{ Institute of Mathematics and Statistics of the University of S\~{a}o Paulo, Rua do Mat\~{a}o, 1010, 05508-090, S\~{a}o Paulo, Brazil. E-mail: yu.pe.petrova@gmail.com.}, Nikita Rastegaev\footnote{St. Petersburg Department of Steklov Mathematical Institute
of Russian Academy of Sciences, 27 Fontanka, 191023, St. Petersburg, Russia. E-mail: rastmusician@gmail.com.}}
\renewcommand{\today}{}
\maketitle
\abstract{
We generalize the previously obtained Kru\v{z}kov-type uniqueness result for the initial-boundary value problem for the chemical flood conservation law system to the case of an almost arbitrary flow function, not restricted by the S-shaped condition or the monotonicity with respect to the chemical agent concentration. The result is applied to the analysis of the Riemann problem solutions for an S-shaped flow function changing monotonicity with respect to the chemical concentration exactly once. All possible Riemann problem solution structures are classified, including certain unique structures that have not been described in earlier studies.

\vspace{5pt}

MSC classes: 35L50 (Primary) 35L65, 35L67, 76L05 (Secondary).
}

\section{Introduction}
We generalize the main result of \cite{MR2024}, where the uniqueness of solutions of the conservation law system
\begin{equation}\label{eq:main_system_chem_flood}
\begin{cases} 
s_t + f(s, c)_x = 0, \\
(cs + a(c))_t + (cf(s,c))_x  = 0,
\end{cases}
\end{equation}
was studied. This system describes the chemical flood of oil reservoir in enhanced oil recovery methods. Here $(x,t)\in\mathbb{R}_+^2$, $s$ is the saturation of the water phase, $c$ is the concentration of the chemical agent dissolved in water, $f$ denotes the fractional flow function, and $a$ describes the adsorption of the chemical agent on the rock, usually concave like the classical Langmuir curve (see Fig.~\ref{fig:BL_ads}b). While in \cite{MR2024} we \mbox{assumed $f$} to be S-shaped (after Buckley--Leverett \cite{BL}) as well as monotone with respect \mbox{to $c$}, in this work we aim to lift those restrictions and consider a more general class of flow functions.

Similarly to \cite{MR2024}, we study the solutions of the initial-boundary value problem
\begin{align}
\label{eq:Initial_boundary_problem}
\begin{split}
&s(x,0) = s^x_0(x), \quad c(x,0) = c^x_0(x), \quad x\geqslant 0,\\
&s(0,t) = s^t_0(t), \quad c(0,t) = c^t_0(t),  \quad t\geqslant 0,
\end{split}
\end{align}
and under certain restrictions on the parameters of the problem and the class of solutions we prove the same uniqueness theorem, that is we prove that two different solutions from the described class with the same initial-boundary data could not exist. 

After that, we will consider a simple case of an S-shaped $f$ changing monotonicity with respect to $c$ exactly once, and classify the solutions to the problem describing constant injection into a homogeneously filled reservoir
\begin{align*}
&s_0^x(x) = s_R, \quad c_0^x(x) = c_R, \quad x\geqslant 0, \\
&s_0^t(t) = s_L, \quad c_0^t(t) = c_L, \quad t\geqslant 0,
\end{align*}
which is equivalent to the Riemann problem
\begin{equation}
\label{eq:Riemann-problem}
    (s,c)(x,0)=
    \begin{cases}
        (s_L,c_L),& \text{if } x\leq 0,\\
        (s_R,c_R),& \text{if } x>0.
    \end{cases}
\end{equation}
Riemann problems are highly important in understanding hyperbolic systems of conservation laws, used in Glimm's random choice method, front tracking methods, etc.~(see \cite{Dan-topological-tool} and references therein for a more comprehensive list of possible applications). The Riemann problem for the system \eqref{eq:main_system_chem_flood} with S-shaped $f$ monotone with respect to $c$ was studied in \cite{JnW} and solutions for it are known. The uniqueness of vanishing viscosity solutions for it was also considered in \cite{Shen}.  For the general case of non-monotone $f$ and multicomponent chromatography \cite{Tveito} provides an algorithm for the construction of Riemann problem solutions. However, they use a different approach to distinguish physically meaningful weak solutions, therefore, our solutions may differ in some cases. Conditions \eqref{eq:Initial_boundary_problem} also cover more complicated problems, including the slug injection problem \cite{PiBeSh06} or tapering \cite{Tapering}. 

The proof of the uniqueness theorem will follow the scheme used in \cite{MR2024} almost exactly. The Lagrange coordinate transformation described in detail in that paper will be utilized, as well as the proof scheme similar to the well-known Kru\v{z}kov's theorem \cite{Kruzhkov}. We will omit proofs that require no changes and only detail the proofs of the lemmas that need significant generalization.

We keep the admissibility criteria used in \cite{MR2024} based on the paradigm of the classical work by Oleinik \cite{Oleinik} with the local variant of the vanishing viscosity condition introduced in \cite{Bahetal} (see (W4) in Definition \ref{def:solution} below). The classification of admissible shocks obtained in \cite{Bahetal} is then used when constructing the Riemann problem solutions. We study all possible Riemann solution structures and describe which values of initial parameters in \eqref{eq:Riemann-problem} yield them. Compared to the monotone case \cite{JnW}, the layout of solution structures in the space of initial parameters is richer and includes areas with novel solution structures not observed before (see \eqref{eq:solution-RP-cscsc} and Figure~\ref{fig:RP-layouts-rare-5}c).

The paper has the following structure. Sect.~\ref{sec:restrictions} lists all restrictions we place on the parameters of the problem for the generalized uniqueness theorem, i.e. on the initial-boundary conditions, on the flow function $f$ and on the adsorption \mbox{function $a$}. Sect.~\ref{sec:admissibility} recalls the definition of the class of admissible solutions and the travelling wave dynamic system for the dissipative system. It also provides a lemma deriving a restriction on the set of admissible shocks. Sect.~\ref{sec:Lagrange} describes the Lagrange coordinate transformation. The qualities of the new flow function are derived similar to the monotone case. Sect.~\ref{sec:entropy} describes the mapping of the shocks in original coordinates and shocks in the Lagrange coordinates. Here, the main change to the proof occurs, as the admissibility conditions for $c$-shocks need to be transferred to the new coordinates and the Kru\v{z}kov-type entropy inequality for them is derived with a slightly different proof. Finally, Sect.~\ref{sec:Riemann-problem} contains the application to the case of an S-shaped $f$ changing monotonicity with respect to $c$ exactly once. All possible solution structures are classified in this section.

\section{Restrictions on problem parameters}
\label{sec:restrictions}

\subsection{Restrictions on the initial-boundary data}
The following restrictions on the functions from \eqref{eq:Initial_boundary_problem} are assumed:
\begin{itemize}
    \item[(S1)] $s^x_0(x) = 0$ for all $x\geqslant x^0$ for a fixed $x^0\in[0,+\infty]$;
    \item[(S2)] if $x^0>0$, then $s^x_0(x) \geqslant \delta^0 > 0$ for all $0 \leqslant x < x^0$;
    \item[(S3)] $s^t_0(t) \geqslant \delta^0 > 0$ for all $t\geqslant 0$.
\end{itemize}

\begin{remark}
\label{remark:c_L_0}
For the sake of including the Riemann problems with $s_L = 0$, the case when $x^0 = \infty$, $s_0^t(t) \equiv 0$ needs to be considered. In this case the zero flow area should be investigated similar to Sect.~\ref{sec:zero-flow}, and is expected to coincide with the axis $\{x=0\}$. Everywhere else the solution will be positive. This creates some special behavior near the vertical axis in Lagrange coordinates, but otherwise we expect the proofs to hold. The details of this case will be considered in some future work.
\end{remark}

\subsection{Restrictions on the flow function}
See Fig.~\ref{fig:BL_ads}a for an example of a function~$f$ under the S-shaped restriction, monotone with respect to $c$. The following assumptions (F1), (F2), (F3') for the fractional flow function $f$ lift those restrictions and allow a much wider class of flow functions. 
\begin{enumerate}
    \item[(F1)] $f\in \mathcal C^2([0,1]^2)$; $f(0, c)=0$, $f(1, c)= 1$ for all $c\in[0,1]$;
    \item[(F2)] $f_s(s, c)>0$ for $0<s<1$, $0 \leq c \leq 1$;  $f_s(0,c)=f_s(1,c)=0$  for all $c\in[0,1]$;
    \item[(F3')] there exists $s^f > 0$, such that $f_{ss}(s, c) > 0$ for all $s\in(0, s^f), c\in[0,1]$.
\end{enumerate}

The assumption (F3') is necessary to maintain some technical steps of the proof of the uniqueness theorem developed in \cite{MR2024} in lieu of the S-shaped assumption (F3) used there. We utilize it exactly once in the proof of Lemma~\ref{lemma:Lax_for_small_s}, which replaces the proof of \cite[Lemma 3.3]{MR2024}.

\subsection{Restrictions on the adsorption function}
The adsorption function $a=a(c)$ satisfies the following assumptions (see Fig.~\ref{fig:BL_ads}b for an example of a function~$a$):
\begin{itemize}
    \item[(A1)] $a \in \mathcal C^2([0,1])$, $a(0) = 0$;
    \item[(A2)] $a_c(c)>0$ for $0<c<1$;
    \item[(A3)] $a_{cc}(c)<0$ for $0<c<1$.
\end{itemize}

\begin{figure}[htbp]
    \centering
    \includegraphics[width=0.4\textwidth]{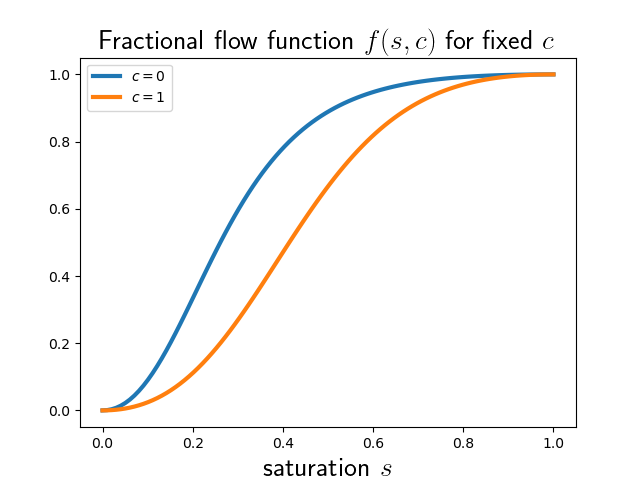}
    \includegraphics[width=0.4\textwidth]{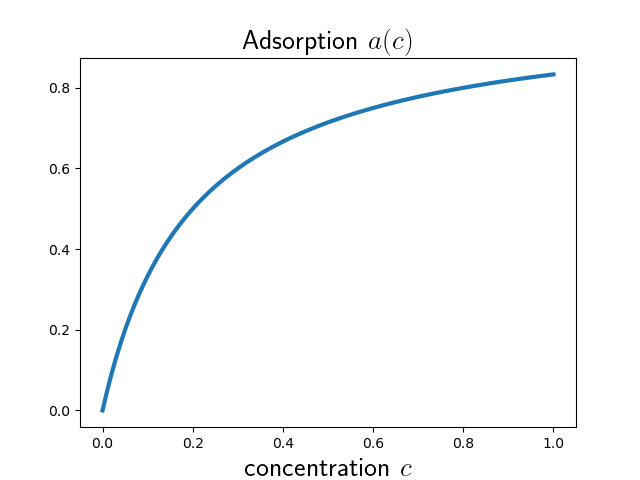}\\
    (a)\qquad\hfil\qquad  
    (b)\hfil
    \caption{Examples of (a) flow function $f(s,c)$;
    (b) adsorption function $a$.}
    \label{fig:BL_ads}
\end{figure}

\section{Admissible solutions of chemical flood system}
\label{sec:admissibility}

\subsection{Admissible weak solutions}
In this work, we use the local vanishing viscosity method proposed in \cite{MR2024}. The admissible weak solutions are required to be a classical solution almost everywhere except for a locally finite number of shocks (jump discontinuities) with a local version of the vanishing viscosity condition on the shocks.

\begin{definition}[Definition 3.1, \cite{MR2024}]
\label{def:solution}
We call $(s, c)$ a piece-wise $\mathcal C^1$-smooth weak solution of \eqref{eq:main_system_chem_flood} with vanishing viscosity admissible shocks and locally bounded ``variation'' of $c$ (\emph{W-solutions} for brevity), if:
\begin{itemize}
    \item[(W1)] Functions $s$ and $c$ are continuous and piecewise continuously differentiable everywhere, except for a locally finite number of $\mathcal C^1$-smooth curves, where one or both of them have a jump discontinuity.
    \item[(W2)] For any compact $K$ away from the axes, the derivative $|c_x(x,t)| < C_K$ is uniformly bounded for all $(x,t) \in K$ not on jump discontinuities.
    \item[(W3)] Functions $s$ and $c$ satisfy \eqref{eq:main_system_chem_flood} in a classical sense inside the areas, where they are continuously differentiable.
    \item[(W4)] On every discontinuity curve $\Gamma$ given by $\gamma(t)$ at any point $(\gamma(t_0), t_0)$ the jump of $s$ and $c$ 
    \[
    s^\pm = s(\gamma(t_0)\pm 0, t_0), \quad c^\pm = c(\gamma(t_0)\pm 0, t_0)
    \]
    with velocity $v = \gamma_t(t_0)$ could be obtained as a limit as $\varepsilon\to 0$ of travelling wave solutions
    \[
    s(x,t) = s\Big(\frac{x - vt}{\varepsilon}\Big), \quad c(x,t) = c\Big(\frac{x - vt}{\varepsilon}\Big)  
    \]
    of the dissipative system
\begin{equation}
\label{eq:main_system_dissipative}
\begin{cases} 
s_t + f(s, c)_x = \varepsilon_c (A(s,c)s_x)_x, \\
(cs + a(c))_t + (cf(s,c))_x  = \varepsilon_c (c A(s,c)s_x)_x+\varepsilon_d c_{xx},
\end{cases}
\end{equation}
with boundary conditions
\[
s(\pm\infty) = s^\pm, \quad c(\pm\infty)  = c^\pm.
\]
Here $\varepsilon_c$ and $\varepsilon_d$ are the dimensionless capillary pressure and diffusion, respectively, and $A(s,c)$ is the capillary pressure function (bounded, separated from zero and Lipschitz continuous).
\end{itemize}
\end{definition}

\begin{remark}
Note that the condition (W1) here is weaker than the condition (W1) in~\cite{MR2024}. It allows $s$ and $c$ to have discontinuities in the derivative. Careful examination of the proofs in \cite{MR2024} shows that this changes nothing and the uniqueness theorem still holds in this wider class of solutions. This modification was recently introduced in \cite{MR2025}, where obvious jumps in $c$ derivative were observed in the slug injection problem. Similarly, we noticed possible jumps in the derivatives at the meeting points of different rarefaction waves in the Riemann problems considered here.
\end{remark}

Note that \eqref{eq:main_system_dissipative} differs from \cite[(4.8)]{JnW}. According to \cite{Bahetal} it yields a different set of admissible shocks in some non-monotone cases, but in the monotone case the admissible shocks are the same. The results of this paper apply to the full system \cite[(3)]{Bahetal}, which accounts for capillary pressure, polymer diffusion and dynamic adsorption, but in \eqref{eq:main_system_dissipative} we chose to disregard the dynamic adsorption for brevity. Note also that \cite{Tveito} does not use the dissipative system to distinguish admissible shocks, and uses the projection principle and the lifting algorithm instead, yielding different admissibility criteria in some cases.

\begin{theorem}
\label{thm:1}
Problem \eqref{eq:main_system_chem_flood} with initial-boundary conditions \eqref{eq:Initial_boundary_problem} satisfying the restrictions (S1)--(S3), with flow function satisfying (F1), (F2), (F3') and adsorption satisfying (A1)--(A3) can only have a unique W-solution.
\end{theorem}

\subsection{Travelling wave dynamic system}
\label{sec:sec2-Hopf}
The assumption (W4) for the solution is that shocks are admissible if and only if they could be obtained as a limit of travelling wave solutions for a system with additional dissipative terms as these terms tend to zero. In this section we analyze such travelling wave solutions and derive a dynamic system that describes them.

Consider a shock between states $(s^-, c^-)$ and $(s^+, c^+)$ moving with velocity $v$. In order to check if it is admissible, we are looking for a travelling wave solution 
\[
s(x,t) = s\Big(\frac{x - vt}{\varepsilon}\Big), \quad c(x,t) = c\Big(\frac{x - vt}{\varepsilon}\Big)
\]
for the dissipative system \eqref{eq:main_system_dissipative}
satisfying the boundary conditions 
\[
s(\pm\infty) = s^\pm, \quad c(\pm\infty)  = c^\pm.
\]
Substituting this travelling wave ansatz into the system \eqref{eq:main_system_dissipative} and denoting $\xi = \frac{x - vt}{\varepsilon}$, we get the system
\begin{equation*}
\begin{cases} 
-v s_\xi + f(s, c)_\xi = s_{\xi\xi}, \\
-v (cs + a(c))_\xi + (cf(s,c))_\xi  = (c s_\xi)_\xi + c_{\xi\xi}.
\end{cases}
\end{equation*}
Integrating the equations over $\xi$ we arrive at the travelling wave dynamic system
\begin{equation}\label{eq:dyn_sys_cap_diff}
\begin{cases} 
s_\xi = f(s, c) - v (s + d_1), \\
c_\xi = v (d_1 c - d_2 - a(c)).
\end{cases}
\end{equation}
The values of $d_1$ and $d_2$ are obtained from the boundary conditions:
\begin{align*}
    vd_1 & = -vs^\pm + f(s^\pm, c^\pm), \\
    vd_2 & = v d_1 c^\pm - v a(c^\pm),
\end{align*}
namely, in the case when $c^+ \neq c^-$,
\begin{equation}\label{eq:d1_d2_def}
d_1 = \dfrac{a(c^-) - a(c^+)}{c^- - c^+}, \quad d_2 =  \dfrac{c^+ a(c^-) - c^- a(c^+)}{c^- - c^+}.
\end{equation}
Additionally, the same boundary conditions yield us the Rankine--Hugoniot conditions
\begin{equation}
\label{eq:RH-1}
\begin{split}
    v[s]&=[f(s,c)],
    \\
    v[cs+a(c)]&=[cf(s,c)],
\end{split}
\end{equation}
where $[q(s,c)]=q(s^+,c^+)-q(s^-,c^-)$\footnote{Note the order of ``$+$'' and ``$-$'' terms in this definition. It could be different in different sources. We follow certain proof schemes of \cite{Serre1}, so our order coincides with their.}. Thus, for every set of shock parameters $(s^{\pm}, c^{\pm})$ and $v$ satisfying \eqref{eq:RH-1}, we can construct a phase portrait for the dynamic system \eqref{eq:dyn_sys_cap_diff}. The points $(s^\pm, c^\pm)$ are critical for this dynamic system due to \eqref{eq:RH-1}, and we can check if there is a trajectory connecting the corresponding critical points. But even just analyzing the geometric meaning of the Rankine--Hugoniot conditions \eqref{eq:RH-1}, we derive a lot of restrictions on admissible shock parameters.

\begin{proposition}\label{prop:inadmissible_shocks}
The following restrictions on admissibility are evident from the properties (F1), (F2), (F3'), (A1)--(A3), the Rankine--Hugoniot conditions \eqref{eq:RH-1} and the analysis of the sign of the right-hand side of \eqref{eq:dyn_sys_cap_diff}:
\begin{itemize}
    \item Admissible shock velocity $v$ is bounded and strictly positive: $0 < v < \|f\|_{\mathcal C^1}$.
    \item Shocks with $s^- = 0$ cannot be admissible.
    \item Shocks with $c^+ > c^-$ cannot be admissible.
    \item If $s^+ = 0$ then $c^+=c^-$.
\end{itemize}
\end{proposition}

\begin{lemma}\label{lemma:Lax_for_small_s}
There exists $s_* \in (0,1)$ such that when $s^- < s_*$, we have
\begin{equation}\label{eq:Lax_for_small_s}
f_s(s^-, c^-) > v
\end{equation}
for all admissible shock parameters.
\end{lemma}
\begin{proof}
If $c^- \neq c^+$, we rewrite \eqref{eq:RH-1} in the following form:
\begin{align*}
    v = \dfrac{[f(s,c)]}{[s]} = \dfrac{f(s^\pm,c^\pm)}{s^\pm + h}, \qquad h = \dfrac{[a(c)]}{[c]},
\end{align*}
therefore, points $(-h, 0)$, $(s^+, f(s^+,c^+))$, $(s^-, f(s^-,c^-))$ are collinear and lie on the line $l(s)=v(s+h)$. Due to (F3') there exists such 
\[
\delta^f = \dfrac{\min \{f(s^f, c): c\in[0,1]\}}{1+\|a\|_{\mathcal{C}^1}}>0,
\]
that when $0 \leqslant v \leqslant \delta^f$, there is a unique intersection point of $l(s)$ and $f(s, c^-)$ inside the interval $(0,1)$. Indeed, by the definition of $\delta^f$ we have $\delta^f (s+h) < f(s, c^-)$ for all $s \geqslant s^f$, therefore there are no intersections on $[s^f, 1]$. And on $(0, s^f)$ there is exactly one intersection due to convexity given by (F3'). We denote this intersection point $(s^c(c^-, v), f(s^c(c^-, v), c^-))$, and at this intersection point we have 
\[
f_s(s^c(c^-, v), c^-) > v.
\]
Therefore, \eqref{eq:Lax_for_small_s} holds for admissible shocks with $c^- \neq c^+$ for any $s^- < s_*^c$, where
\[
s_*^c = \min\limits_{c\in[0,1]} s^c(c, \delta^f).
\]

If $c^-=c^+=c$, we look at the system \eqref{eq:dyn_sys_cap_diff} and note that it simplifies into one equation. A more detailed analysis of this case is given in \cite[Lemma 3.3]{MR2024} and \cite[Sect.~5.2]{MR2024}. We simply note that 
due to Oleinik admissibility we have 
$f_s(s^-, c^-) \geqslant v$. 
To achieve equality the graph of $f(\cdot, c)$ must be tangent to the chord connecting $(s^-, f(s^-, c))$ and $(s^+, f(s^+, c))$. Therefore, denoting by $s^s(c)$ the 
tangent point of the graph of $f(s, c)$ for $s\in(0, s^f)$ and the line going through the point $(1, \min \{f(s^f, c): c\in[0,1]\})$,
and by
\[
s_*^s = \min\limits_{c\in[0,1]} s^s(c),
\]
we conclude that for admissible shocks with $c^-=c^+$ for any $s^- < s_*^s$ we always have a strict sign, as no tangent chord could be constructed there. Thus, \eqref{eq:Lax_for_small_s} holds due to $f$ being convex at $s^-$.
By construction $s^s(c)$ is a positive continuous function on $[0,1]$, 
so it admits a minimum separated from zero.

Finally, denoting $s_* = \min\{ s_*^s, s_*^c \} > 0$, we complete the proof.
\end{proof}

\subsection{Inadmissible nullcline configurations}
\label{sec:sec2-nullclines}
When we get rid of the S-shaped assumption and either the monotonicity or at least the limited non-monotonicity with respect to $c$, it becomes untenable to write down the full classification of possible nullcline configurations, as in \cite[Sect.~4.2]{Bahetal} and \cite[Sect.~3.3]{MR2024}. However, we can still formulate a sufficient restriction on admissible configurations that we use later when transferring admissibility to Lagrange coordinates.

\begin{lemma}
\label{lemma2}
Given fixed $c^+$ and $c^-$, such that $c^+ < c^-$, consider two shocks --- from $(s^-, c^-)$ to $(s^+, c^+)$ with speed $v$ and from $(z^-, c^-)$ to $(z^+, c^+)$ with speed $w$. Assume
\begin{equation}\label{eq:lemma2_assumption}
s^- > z^-, \quad s^+ < z^+, \quad v < w.
\end{equation}
Then both shocks cannot be admissible at the same time.
\end{lemma}

\begin{proof}
From (W4) due to Sect.~\ref{sec:sec2-Hopf} admissible shocks must solve the travelling wave dynamic system \eqref{eq:dyn_sys_cap_diff}. Therefore, if both shocks are admissible, there must exist two trajectories. One trajectory is the function $s(c)$ that solves 
\[
\dfrac{ds}{dc} = \dfrac{f(s, c) - v(s+d_1)}{v(d_1 c - d_2 - a(c))}
\]
and connects $(s^+, c^+)$ to $(s^-, c^-)$, the other is the function $z(c)$ that solves
\[
\dfrac{dz}{dc} = \dfrac{f(z, c) - w(z+d_1)}{w(d_1 c - d_2 - a(c))}
\]
and connects $(z^+, c^+)$ to $(z^-, c^-)$. Note that the values of $d_1$ and $d_2$ are defined in \eqref{eq:d1_d2_def} and depend only on $c^\pm$. Due to the assumption \eqref{eq:lemma2_assumption} the graphs of these function $s(c)$ and $z(c)$ must intersect, and the $s(c)$ graph must cross the $z(c)$ graph from below at one of the intersections. Denote such intersection $(\bar{s}, \bar{c})$. At that point we necessarily have
\[
\dfrac{f(\bar{s}, \bar{c}) - v(\bar{s}+d_1)}{v(d_1 \bar{c} - d_2 - a(\bar{c}))} \geqslant \dfrac{f(\bar{s}, \bar{c}) - w(\bar{s}+d_1)}{w(d_1 \bar{c} - d_2 - a(\bar{c}))}.
\]
Note that $d_1 \bar{c} - d_2 - a(\bar{c}) < 0$ due to (A3), and we arrive at $v \geqslant w$, which contradicts~\eqref{eq:lemma2_assumption}. Thus, at least one of the shocks is inadmissible.
\end{proof}

\section{Lagrange coordinate transformation}
\label{sec:Lagrange}

\subsection{Zero flow area}
\label{sec:zero-flow}
We would like to utilize the Lagrange coordinates, i.e. the coordinates tied to the flow, but to use them rigorously, we first need to establish the area, where there is no flow. Due to (F1) and (F2), the flow function is zero only when $s=0$. In this section we prove that under the conditions (S1)--(S3) it is a connected area in $Q = \mathbb{R}_+^2$ in $(x,t)$ space bounded by the ray $(x^0, +\infty)$ on one side and by a curve made of discontinuities of the function $s$ on the other.
\begin{lemma}[Lemma 4.1, \cite{MR2024}]
\label{lemma_zero_below}
Let $(x_*, t_*)$, $x_*>0$, $t_*>0$ be a zero of the solution $s$, i.e. $s(x_*, t_*) = 0$ in the smoothness area or one of $s(x_*\pm0, t_*) = 0$ on the shock. Then $s(x_*, t) = 0$ for all $0 \leqslant t < t_*$.
\end{lemma}

\begin{lemma}[Lemma 4.4, \cite{MR2024}]
\label{lemma:t0_is_shocks}
For all $x > x^0$ we define
\[
t_0(x) = \sup\{ t: s(x, t) = 0 \}.
\]
Then 
\begin{itemize}
    \item $t_0(x) < +\infty$;
    \item $(x, t_0(x))$ is a point on a shock;
    \item $t_0(x)$ is continuous, piece-wise $\mathcal{C}^1$-smooth.
\end{itemize}
\end{lemma}

\begin{corollary}[Corollary 4.5, \cite{MR2024}]
\label{corollary:zero_flow_area}
Define $\Omega_0 = \{(x,t): x > x^0, 0\leqslant t < t_0(x)\}$. Then $s(x,t) = 0$ in $\Omega_0$ and $s(x,t) > 0$ outside $\overline{\Omega}_0$. Moreover, $s(x,t)$ is locally separated from $0$ outside $\overline{\Omega}_0$.
\end{corollary}

\begin{proposition}[Proposition 4.6, \cite{MR2024}]
\label{prop:c_on_zero_boundary}
$c_t = 0$ in $\Omega_0$, therefore $c(x, t_0(x)) = c^x_0(x)$.
\end{proposition}

\subsection{Lagrange coordinates}
\label{sec2-Lagrange}
It is hard to trace back the history of the coordinates transformation described in this subsection. Many authors describe the transformation with no citations of previous work. The oldest reference we found is \cite{Courant} cited in \cite{Wa87} in the context of gas dynamics equations. The idea is also presented in the lectures by Gelfand \cite{Gelfand} for the case of an arbitrary system of conservation laws. The splitting technique for the system \eqref{eq:main_system_chem_flood} using the Lagrange coordinate transformation is presented in \cite{PiBeSh06}. It is later developed and applied to different systems by many authors (see \cite{Pires2021} and references therein).

\begin{proposition}[Proposition 4.3, \cite{MR2024}]
\label{prop:contour_integral}
For any solution $(s,c)$ the differential form $f(s,c)\,dt-s\,dx$ derived from the first equation of \eqref{eq:main_system_chem_flood} is exact, i.e. on any closed curve $\partial \Omega$ with finite number of shock points we have
\begin{equation}\label{eq:contour_int}
\oint\limits_{\partial \Omega} f(s,c)\,dt - s\, dx = 0.
\end{equation}
Similarly, from the second equation of \eqref{eq:main_system_chem_flood} we derive the exact form $(cs + a(c)) dx - f(s,c) dt$, therefore
\begin{equation}\label{eq:contour_int2}
    \oint\limits_{\partial\Omega} c (s\, dx - f(s,c)\, dt) + \oint\limits_{\partial\Omega} a(c)\, dx = 0.
\end{equation}
\end{proposition}

We denote by $\varphi$ the potential such that
\begin{equation} 
\label{eq:dPhi}
    d\varphi=f(s,c)\,dt-s\,dx.
\end{equation}
To explain the physical meaning of $\varphi$ let us consider any trajectory $\nu$ connecting $(0, 0)$ and $(x, t)$.
When $s$ denotes the saturation of some liquid, the potential $\varphi(x, t)$ is equal to the amount of this liquid
passing through the trajectory:
\begin{equation}
\label{eq:phi}
    \varphi(x,t)=\int\limits_{\nu} f(s,c)\,dt-s\,dx.
\end{equation}
The coordinate change $(x, t) \to (\varphi, x)$ is only applicable in the area $Q_{orig} = Q\setminus \overline{\Omega}_0$, where the saturation $s$ and the flow function $f(s,c)$ are not zero. It keeps the $x$ coordinate, so it maps the axis $\Gamma_t = \{x=0, t\geqslant 0\}$ onto $\Gamma_\varphi = \{\varphi\geqslant 0, x=0\}$. 
The segment $[0, x^0]\times\{0\}$ maps into a curve $( \varphi_0(x), x)$, where
\begin{equation}
\label{eq:def_varphi_0}
\varphi_0(x) = -\int\limits_0^x s^x_0(r) \, dr.
\end{equation}
The curve $(x, t_0(x))$ for $x > x^0$ maps into a horizontal line beginning at the point $(\varphi_0(x^0), x^0)$. Therefore $Q_{orig}$ maps into 
\[
Q_{lagr} = Q \cup \{ (\varphi, x) : 0 < x \leqslant x^0, 0 > \varphi > \varphi_0(x) \} \cup \left((\varphi_0(x^0), 0) \times (x^0, +\infty)\right),
\]
see Fig.~\ref{fig:orig-lagr-areas}.

Corollary \ref{corollary:zero_flow_area} guaranties that there is a reverse transform given by
\[
dt=\frac1{f(s,c)}\,d\varphi+\frac s{f(s,c)}\,dx,
\]
and the denominators are locally separated from zero, therefore this coordinate change is a piecewise $\mathcal C^1$-diffeomorphism.
Since $\mathcal C^1$-smooth curves preserve their smoothness properties under any diffeomorphism, all discontinuity curves map into $\mathcal C^1$-smooth discontinuity curves. 

Substituting \eqref{eq:dPhi} into \eqref{eq:contour_int2}, we get that the form $  - c \,d \varphi + a(c) \,dx$ is also exact in any area in $Q_{lagr}$ where the images of $(s, c)$ are $\mathcal C^1$-smooth. Therefore, it is closed too and leads to the identity
\[
0 = d(c \,d \varphi - a(c)\, dx) = \left( \frac{\partial c}{\partial x} + \frac{\partial a(c)}{\partial \varphi}\right)\, dx \wedge d\varphi.
\]
Together with the identity
\[
0 = d(dt) = d\left( \dfrac{1}{f} \, d\varphi + \dfrac{s}{f} \, dx\right) = \left(\frac{\partial}{\partial x}\left(\frac1f\right) - \frac{\partial}{\partial \varphi}\left(\frac s f\right)\right) \, dx \wedge d\varphi
\]
it gives us inside the areas of $\mathcal C^1$-smoothness the classical system
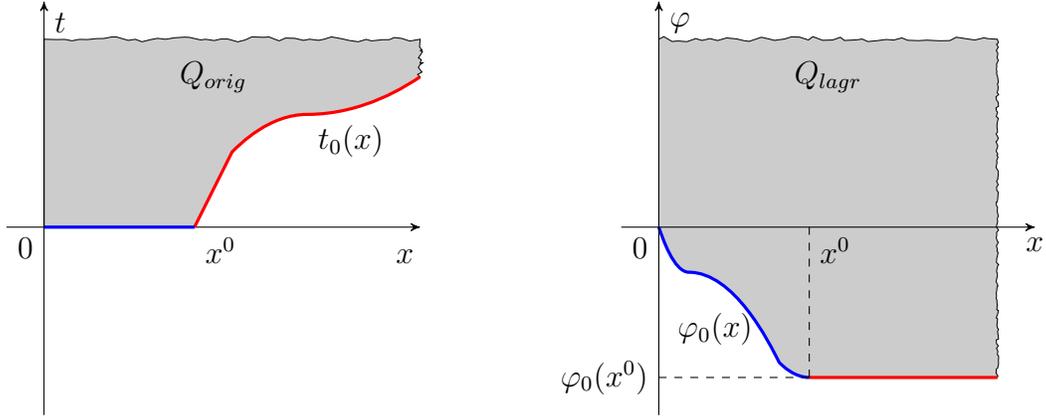
\begin{figure}
    \begin{minipage}{\linewidth}
    \begin{center}
    \def\lowborder{-5}
\def\leftborder{-1}
\def\upperborder{6}
\def\rightborder{10}
\def\startpoint{\rightborder * 0.4}
\def\phinote{0.8 * \lowborder}
\def\eps{1}
\begin{tikzpicture}[>=stealth', yscale=0.5, xscale=0.5]
\draw[thin,->] ({\leftborder}, 0) -- (\rightborder, 0) node[below] {$x\phantom{x^0}$};
\begin{scope}
\filldraw[fill=black!20, ultra thin]
(\startpoint, 0)  -- (\startpoint+1, 2) parabola[bend at end] (\startpoint + 3, 3) parabola (\startpoint + 6, 4)
  decorate [decoration={random steps,segment length=1pt,amplitude=1pt}] {-- (\startpoint + 6, \upperborder - \eps)}
 decorate [decoration={random steps,segment length=3pt,amplitude=1pt}] {--(0, \upperborder - \eps)}
 -- (0, 0) -- cycle;
\draw[very thick, red] (\startpoint, 0)  -- (\startpoint+1, 2) parabola[bend at end] (\startpoint + 3, 3) parabola (\startpoint + 6, 4);
\node[below right] at (\startpoint + 3, 3) {$t_0(x)$} ;
\node[] at (\rightborder * 0.45, 4) {$Q_{orig}$};
\node[below right] at (\startpoint, 0) {$x^0$};
\node[below left] at (0, 0) {$0$};
\draw[very thick, blue] (0, 0) -- (\startpoint, 0);
\draw[thin,->] (0,\lowborder) -- (0,\upperborder) node[below right ] {$t$};
\end{scope}
\end{tikzpicture}
\hspace{1cm}
\begin{tikzpicture}[>=stealth', yscale=0.5, xscale=0.5]
\begin{scope}
\filldraw[fill=black!20, ultra thin]
(0, 0)  parabola[bend at end] (\startpoint * 0.2, \phinote * 0.3)
parabola (\startpoint * 0.8, \phinote *0.9) parabola[bend at end] (\startpoint, \phinote) -- (\rightborder - \eps, \phinote)
  decorate [decoration={random steps,segment length=1pt,amplitude=0.5pt}] {-- (\rightborder - \eps, \upperborder - \eps)}
 decorate [decoration={random steps,segment length=3pt,amplitude=1pt}] {--(0, \upperborder - \eps)}
 -- (0, 0) -- cycle;
 \draw[thin,->] ({\leftborder}, 0) -- (\rightborder, 0) node[below] {$x$};
\draw[very thick, blue] (0, 0)  parabola[bend at end] (\startpoint * 0.2, \phinote * 0.3)
parabola (\startpoint * 0.8, \phinote *0.9) parabola[bend at end] (\startpoint, \phinote);
\node[] at (\rightborder * 0.45, 4) {$Q_{lagr}$};
\node[below right] at (\startpoint, 0) {$x^0$};
\node[below left] at (0, 0) {$0$};
\draw[very thick, red] (\startpoint, \phinote) -- (\rightborder - \eps, \phinote);
\draw[dashed]  (\startpoint, \phinote) -- (\startpoint, 0);
\draw[dashed] (\startpoint, \phinote) -- (0, \phinote);
\node[left] at (0, \phinote) {$\varphi_0(x^0  )$};
\node[below] at (\startpoint/2-0.5, \phinote/2) {$\varphi_0(x)$};
\end{scope}
\draw[thin,->] (0,\lowborder) -- (0,\upperborder) node[below right ] {$\varphi$};
\end{tikzpicture}
    \end{center}
    \end{minipage}
    \caption{Areas $Q_{orig}$ (grey area on the left), and $Q_{lagr}$ (grey area on the right). The red curve on the left is mapped onto the red line on the right. The blue line on the left is mapped onto the blue curve on the right.}
    \label{fig:orig-lagr-areas}
\end{figure}
\begin{align*}
\begin{split}
    \frac{\partial}{\partial x}\left(\frac1f\right) - \frac{\partial}{\partial \varphi}\left(\frac s f\right)&=0, \\
    \frac{\partial c}{\partial x} + \frac{\partial a(c)}{\partial \varphi}&=0. 
\end{split}
\end{align*}
We use the notation
\begin{equation} 
\label{eq:def-U-F}
\mathcal{U}(\varphi, x)=\frac1{f(s(x,t),c(x,t))}, \quad \zeta(\varphi,x) = c(x,t), \quad\text{and}\quad\mathcal F(\mathcal U,\zeta)=-\frac{s}{f(s,c)}
\end{equation}
in order to transform this system into the system of conversation laws
\begin{align}
       \mathcal U_x + \mathcal F(\mathcal U,\zeta)_\varphi & = 0,\label{eq:U-lagr-eqn}\\
    \zeta_x + a(\zeta)_\varphi&=0. \label{eq:c-lagr-eqn}
\end{align}

\begin{proposition}[Proposition 4.7, \cite{MR2024}]
Suppose that $(s,c)$ is a W-solution of the problem \eqref{eq:main_system_chem_flood}  (see Definition~\ref{def:solution})
with initial and boundary conditions \eqref{eq:Initial_boundary_problem} satisfying the conditions (S1)--(S3). Then the functions $(\mathcal U, \zeta)$ given by the formulae \eqref{eq:def-U-F} satisfy the integral equations
\begin{align}
\label{eq:U_weak}
\begin{split}
    \iint\limits_{Q_{lagr}} \mathcal U \widetilde \psi_x + \mathcal F(\mathcal U, \zeta) \widetilde \psi_\varphi  & \,d\varphi\,dx +
    \int\limits_0^\infty\mathcal U_0^\varphi(\varphi) \widetilde \psi(\varphi, 0)\,d\varphi \\ 
    & + \int\limits_{x^0}^\infty
 \mathcal F(\mathcal U(\varphi_0(x^0)+0, x), c_0^x(x))\widetilde \psi(\varphi_0(x^0), x)\,dx = 0
\end{split}
\end{align}
and
\begin{align}
\label{eq:zeta_weak}
\begin{split}
    \iint\limits_{Q_{lagr}} \zeta\widetilde\psi_x + a(\zeta) \widetilde\psi_\varphi & \,d\varphi\,dx + 
    \int\limits_0^{x^0}  (s_0^x(x)c_0^x(x) + a(s_0^x(x))) \widetilde\psi(\varphi_0(x), x)\,dx \\
    {} & + \int\limits_{x^0}^\infty  a(c_0^x(x)) \widetilde\psi(\varphi_0(x^0), x)\,dx +
    \int\limits_{0}^\infty \zeta_0^\varphi(\varphi)\widetilde\psi(\varphi, 0) \, d\varphi  = 0
\end{split}
\end{align}
for all test function $\widetilde \psi \in \mathcal D(Q_{lagr})$,
where the initial values $\mathcal U_0^\varphi$ and $\zeta_0^\varphi$ are given by 
\begin{equation}
\label{eq:U_zeta_weak_initial_values}
\mathcal U_0^\varphi(\varphi(0, t)) = \dfrac{1}{f(s_0^t(t), c_0^t(t))}, \qquad \zeta_0^\varphi (\varphi(0, t)) = c_0^t(t).
\end{equation}
\end{proposition}

\begin{remark}
The same reasoning as in \cite[Lemma 2.2.1]{Serre1} shows that on every shock the equations in the weak form \eqref{eq:U_weak} and \eqref{eq:zeta_weak} result in the Rankine--Hugoniot condition
\begin{equation} 
\label{eq:RH-2}
\begin{split}
    v^*[\mathcal U]&=[\mathcal F(\mathcal U,\zeta)],
    \\
    v^*[\zeta]&=[a(\zeta)],
\end{split}
\end{equation} 
where $v^*$ is the velocity of the shock between states $(\mathcal U^-, \zeta^-)$ and $(\mathcal U^+, \zeta^+)$. Here, like in original coordinates, $[q(\mathcal U, \zeta)]=q(\mathcal U^+, \zeta^+)-q(\mathcal U^-, \zeta^-)$.
\end{remark}

Properties of the new flow function $\mathcal F$ (see Fig.~\ref{fig:BLf_lagr}) that correspond to the properties (F1), (F2) of the function $f$ are listed below.
\begin{figure}[htbp]
    \centering
    \includegraphics[width=0.55\textwidth]{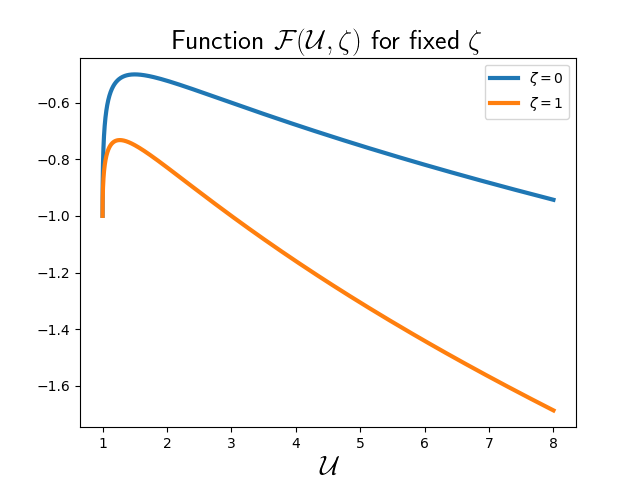}
    \caption{The function $\mathcal F(\mathcal U,\zeta)$ corresponding to the flow function $f(s, c)$ plotted in Fig.~\ref{fig:BL_ads}a.}
    \label{fig:BLf_lagr}
\end{figure}
\begin{proposition}
\label{prop:new-flow-function-properties}
For all $\zeta\in[0,1]$ the following properties of the function $\mathcal F$ are fulfilled
\begin{itemize}
\item[$\bullet$]
$\mathcal F \in \mathcal C^2([1,+\infty)\times[0,1])$; 
\item[$\bullet$]
$\mathcal F(\mathcal U,\zeta) < 0$ for all $\,\mathcal U\in [1,+\infty)$;
\item[$\bullet$]
$\mathcal F(1, \zeta) = - 1$;
\item[$\bullet$]
$\lim\limits_{\mathcal U\to\infty}\mathcal{F}(\mathcal U, \zeta) = -\infty$;
\item[$\bullet$]
$\lim\limits_{\mathcal U\to1}\mathcal{F}_{\mathcal U}(\mathcal U, \zeta) = +\infty$;
\item[$\bullet$]
$\lim\limits_{\mathcal U\to\infty}\mathcal{F}_\mathcal{U}(\mathcal U, \zeta) = 0$.
 \end{itemize}
\end{proposition}

\section{Entropy conditions in Lagrange coordinates}
\label{sec:entropy}

\subsection{Mapping shocks to Lagrange coordinates}
Let us recall the notation from \cite{MR2024} pertaining to the mapping of shocks between the coordinate systems.

Consider a shock in original coordinates at the point $(x_1, t_1)$ with values
\[
s^\pm = s(x_1\pm 0, t_1), \quad c^\pm = c(x_1\pm 0, t_1).
\]
Denote $\theta_c(s) = \frac{1}{f(s,c)}$ and $\vartheta_c = \theta_c^{-1}$ its inverse function with respect to its argument~$s$. Using these functions we map
\[
\mathcal U^{[+]} = \theta_{c^+}(s^+), \quad \mathcal U^{[-]} = \theta_{c^-}(s^-), \quad \zeta^{[+]} = c^+, \quad \zeta^{[-]} = c^-,
\]
\[s^+ = \vartheta_{\zeta^{[+]}}(\mathcal U^{[+]}), \quad s^- = \vartheta_{\zeta^{[-]}}(\mathcal U^{[-]}), \quad c^+ = \zeta^{[+]}, \quad c^- = \zeta^{[-]}. 
\]
Note, that in the original coordinates values $s^\pm$ correspond to $x\to x_1\pm 0$ respectively. The shock velocity in original coordinates is always positive due to Proposition \ref{prop:inadmissible_shocks}, therefore, $s^\pm$ correspond to $t\to t_1\mp 0$:
\[
s^\pm = s(x_1, t_1 \mp 0), \quad c^\pm = c(x_1, t_1\mp 0).
\]
Further, due to \eqref{eq:dPhi} and (F1), when $x = x_1$ is fixed, $t\to t_1\mp 0$ correspond to $\varphi\to \varphi_1\mp 0$ for the point $(\varphi_1, x_1)$ on a corresponding shock in Lagrange coordinates, and so do $\mathcal U^{[\pm]}$:
\[
\mathcal U^{[\pm]} = \mathcal U(\varphi_1 \mp 0, x_1), \quad \zeta^{[\pm]} = \zeta(\varphi_1 \mp 0, x_1).
\]
But for the equations \eqref{eq:U-lagr-eqn}, \eqref{eq:c-lagr-eqn} in Lagrange coordinates the $x$ axis plays the role of time and $\varphi$ the role of space, so we would like to denote $\mathcal U^\pm$ to correspond to $\varphi\to \varphi_1\pm 0$. Thus, we denote
\[
\mathcal U^+ = \mathcal U^{[-]}, \quad \mathcal U^- = \mathcal U^{[+]}, \quad \zeta^+ = \zeta^{[-]}, \quad \zeta^- = \zeta^{[+]},
\]
\[
\mathcal U^{\pm} = \mathcal U(\varphi_1 \pm 0, x_1), \quad \zeta^{\pm} = \zeta(\varphi_1 \pm 0, x_1),
\]
and obtain a one-to-one mapping of shocks in original $(s^\pm, c^\pm)$ and Lagrange $(\mathcal U^\pm, \zeta^\pm)$ coordinates:
\begin{equation}\label{shock_mapping}
\begin{split}
(\mathcal U^\pm, \zeta^\pm) &\to (s^\pm, c^\pm) = (\vartheta_{\zeta^\mp}(\mathcal U^\mp), \zeta^\mp), \\
(s^\pm, c^\pm) &\to (\mathcal U^\pm, \zeta^\pm) = (\theta_{c^\mp}(s^\mp), c^\mp).
\end{split}
\end{equation}

\subsection{Oleinik, Lax and entropy admissibility for $\mathcal U$-shocks}
\label{sec:sec2-s-shock-admissibility}
Recall from \cite{MR2024} the following lemma and a proposition that follows from it. Its proof does not rely on the properties we eliminated. We repeat the same argument as \cite[Sect.~5.2]{MR2024} to transfer the Oleinik's E-condition to the Lagrange coordinates and prove that it implies the entropy condition for any convex positive entropy, but in particular for entropy--entropy-flux pairs
$(|\mathcal U - k|, \mathcal G(\mathcal U, k))$, $k\in \mathbb{R}$, where
\[
\mathcal G(\mathcal U, k) = (\mathcal F(\mathcal U, \zeta) - \mathcal F(k, \zeta))\sign(\mathcal U - k).
\]

\begin{lemma}[Lemma 5.4, \cite{MR2024}]
\label{lemma-U-entropy-int-ineq}
For all $k\in\mathbb{R}$ and every positive test function $\psi \in \mathcal D^+(Q_{lagr})$ with $\supp \psi$ containing only $\mathcal U$-shocks (and no $\zeta$-shocks), we have the entropy condition
\begin{align}
\label{entropy_integral_inequality}
\begin{split}
0 \leqslant &\iint\limits_{Q_{lagr}} |\mathcal U - k|\psi_x + \mathcal G(\mathcal U, k) \psi_\varphi \,d\varphi \,dx - \iint\limits_{Q_{lagr}} \mathcal F_\zeta(k, \zeta) \zeta_\varphi \sign(\mathcal U-k) \psi \,d\varphi \,dx \\
&{} + \int\limits_0^{x^0} \Big( |\mathcal U(\varphi_0(x), x) - k| s_0^x(x) + \mathcal G(\mathcal U(\varphi_0(x), x), k)\Big) \psi(\varphi_0(x), x) \, dx \\
&{} + \int\limits_{x^0}^\infty \mathcal G(\mathcal U(\varphi_0(x^0)+0, x), k)\psi(\varphi_0(x^0), x)\,dx + \int\limits_0^\infty |\mathcal U(\varphi, 0) - k|\psi(\varphi, 0)\,d\varphi.
\end{split}
\end{align}
\end{lemma}

\begin{proposition}[Proposition 2.7.1, \cite{Serre1}]
\label{proposition_entropy_solution_inequality}
For a fixed solution $\zeta$, given two solutions $\mathcal U$ and $\mathcal V$, for any positive test function $\psi \in \mathcal D^+(Q_{lagr})$ with $\supp \psi$ containing no $\zeta$-shocks, we have
\begin{align}
\label{entropy_solutions_inequality}
\begin{split}
0 \leqslant &\iint\limits_{Q_{lagr}} |\mathcal U - \mathcal V|\psi_x + \mathcal G(\mathcal U, \mathcal V) \psi_\varphi \,d\varphi \,dx \\
&{} + \int\limits_0^{x^0} |\mathcal U(\varphi_0(x), x) - \mathcal V(\varphi_0(x), x)| s_0^x(x) \psi(\varphi_0(x), x) \, dx \\
&{} + \int\limits_0^{x^0} \mathcal G(\mathcal U(\varphi_0(x), x), \mathcal V(\varphi_0(x), x)) \psi(\varphi_0(x), x) \, dx \\
&{} + \int\limits_{x^0}^\infty \mathcal G(\mathcal U(\varphi_0(x^0)+0, x), \mathcal V(\varphi_0(x^0)+0,x))\psi(0, x)\,dx \\
&{} + \int\limits_0^\infty |\mathcal U(\varphi, 0) - \mathcal V(\varphi, 0)|\psi(\varphi, 0)\,d\varphi.
\end{split}
\end{align}
\end{proposition}

In \cite{Serre1} this Proposition is proved for the case that corresponds to $\zeta \equiv const$. For the general case, see the original proof of Kru\v{z}kov's \cite[Theorem 1]{Kruzhkov}.

\subsection{Oleinik and Lax admissibility of $\zeta$-shocks}
Similarly, $\zeta$-shocks in the second equation do not rely on any properties of the flow function, therefore the following proposition holds with no changes.
\begin{proposition}[Proposition 5.8, \cite{MR2024}]
\label{prop-zeta-entropy-integral}
Denote $\mathcal A(\zeta, k) = (a(\zeta) - a(k))\sign(\zeta-k)$. Then on any admissible $\zeta$-shock given by the curve $(\Phi(x), x)$ inside $Q_{lagr}$ we have the entropy inequality
\[
[\mathcal A(\zeta, k)] \leqslant \dfrac{d\Phi}{dx}[|\zeta-k|], \quad k\in\mathbb{R},
\]
and therefore for every positive test function $\psi \in \mathcal D^+(Q_{lagr})$ we have the integral entropy condition
\begin{align*}
0 \leqslant &\iint\limits_{Q_{lagr}} |\zeta - k|\psi_x + \mathcal A(\zeta, k)\psi_\varphi \,d\varphi\,dx \\
&{} + \int\limits_0^{x^0} \Big( |\zeta(\varphi_0(x), x) - k| s_0^x(x) + \mathcal A(\zeta(\varphi_0(x), x), k)\Big) \psi(\varphi_0(x), x) \, dx \\
{} & + \int\limits_{x^0}^\infty \mathcal A(\zeta(\varphi_0(x^0),x), k)\psi(\varphi_0(x^0),x) \,dx + \int\limits_0^\infty |\zeta(\varphi, 0) - k|\psi(\varphi, 0)\, d\varphi.
\end{align*}
\end{proposition}

The first and only significant change to the uniqueness proof comes when we look at the proof of \cite[Lemma 9]{MR2024}. It relies on the admissible shocks' classification, and needs to be reworked based solely on Lemma \ref{lemma2}.

\begin{lemma}
\label{lemma_zeta_shock_entropy_condition}
Given two solutions $(\mathcal U, \zeta)$ and $(\mathcal V, \zeta)$ for the same $\zeta$, we have the following entropy inequality on any $\zeta$-shock given by the curve $(\Phi(x), x)$ inside $Q_{lagr}$:
\begin{equation}\label{eq:entropy-ineq-zeta-shock}
[\mathcal G(\mathcal U, \mathcal V)] \leqslant \dfrac{d\Phi}{dx} [|\mathcal U - \mathcal V|].
\end{equation}
\end{lemma}
\begin{proof}
Let the two solutions have the $\zeta$-shocks, that, when translated into the original coordinates via the mapping \eqref{shock_mapping}, connect 
\begin{align*}
&(s^- = \vartheta_{\zeta^+}(\mathcal{U}^+), c^- = \zeta^+) \text{ to } (s^+ = \vartheta_{\zeta^-}(\mathcal{U}^-), c^+ = \zeta^-) \text{ with speed } v, \text{ and }
\\ 
&(z^- = \vartheta_{\zeta^+}(\mathcal{V}^+), c^- = \zeta^+) \text{ to } (z^+ = \vartheta_{\zeta^-}(\mathcal{V}^-), c^+ = \zeta^-) \text{ with speed } w.
\end{align*}

First, we consider the case $v=w$. In this case we have equality in \eqref{eq:entropy-ineq-zeta-shock} due to the Rankine--Hugoniot condition (see \eqref{eq:RH-2}):
\begin{equation}
\label{Lagrangian_Rankine_Hugoniot}
\dfrac{d\Phi}{dx} = \dfrac{[\mathcal F(\mathcal U, \zeta)]}{[\mathcal U]} = \dfrac{[\mathcal F(\mathcal V, \zeta)]}{[\mathcal V]},   
\end{equation}
since all points lay on the same line in Figure \ref{fig:two_nullclines_hodograph}, and therefore, we additionally have
\[
\dfrac{d\Phi}{dx} = \dfrac{\mathcal F(\mathcal U^\pm, \zeta^\pm) - \mathcal F(\mathcal V^\pm, \zeta^\pm)}{\mathcal U^\pm - \mathcal V^\pm}.
\]

Next, we can assume without loss of generality that $v<w$. Then the relations between $s^\pm$ and $z^\pm$ fall into one of four cases:

\begin{itemize}
    \item $s^- > z^-$, $s^+ < z^+$. This gives the assumption \eqref{eq:lemma2_assumption} for our shocks, thus due to Lemma \ref{lemma2} one of the considered shocks must be inadmissible. Therefore, we don't need to consider this case.
    \item $s^- > z^-$, $s^+ > z^+$ or $s^- < z^-$, $s^+ < z^+$. In these cases we have 
    \[
    \sign(\mathcal U^+ - \mathcal V^+) = \sign(\mathcal U^- - \mathcal V^-),
    \]
    therefore we have equality in \eqref{eq:entropy-ineq-zeta-shock} due to Rankine--Hugoniot condition \eqref{Lagrangian_Rankine_Hugoniot}.
    \item $s^- < z^-$, $s^+ > z^+$. In this case 
    \begin{align}
    \begin{split}
    \label{eq:known_signs}
    \sign(\mathcal U^+ - \mathcal V^+) = \sign(\theta_{c^-}(s^-) - \theta_{c^-}(z^-)) = 1, \\ \quad \sign(\mathcal U^- - \mathcal V^-)= \sign(\theta_{c^+}(s^+) - \theta_{c^+}(z^+)) = -1,
    \end{split}
    \end{align}
    so \eqref{eq:entropy-ineq-zeta-shock} transforms into
    \[
    \mathcal F(\mathcal U^+, \zeta^+) - \mathcal F(\mathcal V^+, \zeta^+) + \mathcal F(\mathcal U^-, \zeta^-) - \mathcal F(\mathcal V^-, \zeta^-) \leqslant \dfrac{d\Phi}{dx} (\mathcal U^+ - \mathcal V^+ + \mathcal U^- - \mathcal V^-).
    \]
    Note that due to the monotonicity of $a$ and the Rankine-Hugoniot condition \eqref{eq:RH-2} the speed of the shock
    \begin{equation}
    \label{eq:zeta-shock-velocity-positive}
    \dfrac{d\Phi}{dx} = \dfrac{a(\zeta^-)-a(\zeta^+)}{\zeta^--\zeta^+} > 0.
    \end{equation}
    It could be observed geometrically on Fig.~\ref{fig:two_nullclines_hodograph} that
    \begin{align*}
    \mathcal U^+ = \theta_{c^-}(s^-) > \theta_{c^-}(z^-) = \mathcal V^+, \\
    \mathcal U^- = \theta_{c^+}(s^+) < \theta_{c^+}(z^+) = \mathcal V^-,
    \end{align*}
    therefore, the following inequalities hold for the inclines of lines:
    \[
    \dfrac{\mathcal F(\mathcal U^+, \zeta^+) - \mathcal F(\mathcal V^+, \zeta^+)}{\mathcal U^+ - \mathcal V^+} < \dfrac{d\Phi}{dx},
    \]
    \[
    \dfrac{\mathcal F(\mathcal U^-, \zeta^-) - \mathcal F(\mathcal V^-, \zeta^-)}{\mathcal U^- - \mathcal V^-} > \dfrac{d\Phi}{dx}.
    \]
    This Fig.~\ref{fig:two_nullclines_hodograph} is constructed for the S-shaped case, but the geometrical argument holds even without that assumption.
    \begin{figure}[t!]
    \begin{center}
    \includegraphics[trim={0.1cm 0.1cm 0.1cm 0.1cm}, clip, width=0.8\linewidth]{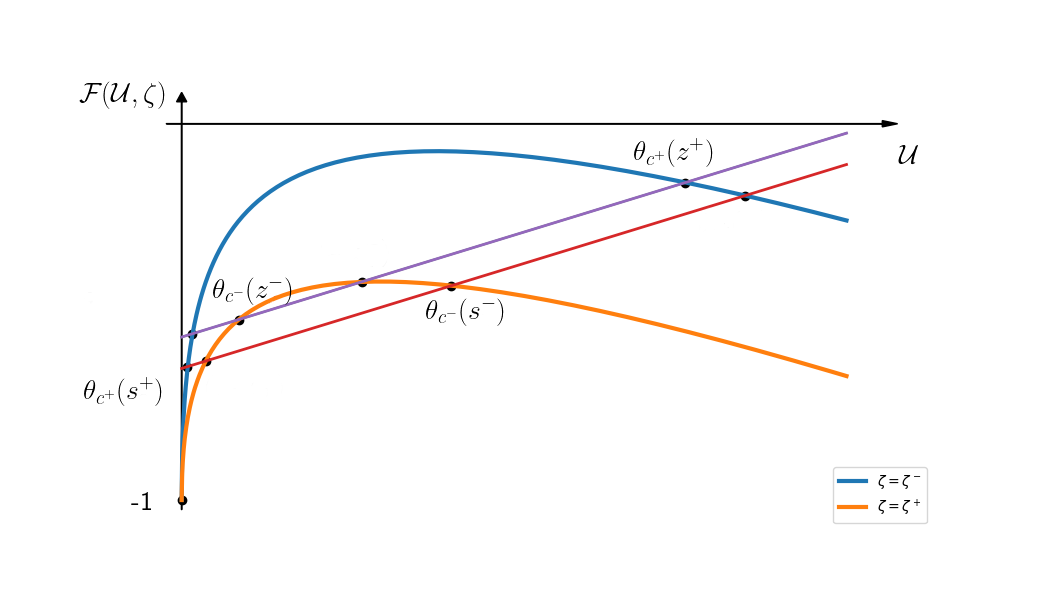}
    \end{center}
    \caption{An illustration for the positions of $\mathcal U^\pm$ and $\mathcal V^\pm$ on the plots of $\mathcal F(\cdot, \zeta^\pm)$} in the last case of Lemma \ref{lemma_zeta_shock_entropy_condition}. The red line corresponding to $s^\pm$ is lower due to $v<w$.
    \label{fig:two_nullclines_hodograph}
\end{figure}
    Taking into account the known signs \eqref{eq:known_signs} of the denominators, we obtain
     \[
    \mathcal F(\mathcal U^+, \zeta^+) - \mathcal F(\mathcal V^+, \zeta^+) < \dfrac{d\Phi}{dx} (\mathcal U^+ - \mathcal V^+),
    \]
    \[
    \mathcal F(\mathcal U^-, \zeta^-) - \mathcal F(\mathcal V^-, \zeta^-) < \dfrac{d\Phi}{dx} (\mathcal U^- - \mathcal V^-).
    \]
    Taking the sum of these inequalities concludes this case.
\end{itemize}
Thus, all four cases lead either to contradiction or the necessary inequality, which concludes the proof.
\end{proof}

Thus, with \cite[Lemma 9]{MR2024} successfully replaced, we can proceed with the original proof with no other changes. The following lemma follows from Proposition~\ref{proposition_entropy_solution_inequality} and the new Lemma~\ref{lemma_zeta_shock_entropy_condition} we just proved as a replacement for \cite[Lemma 9]{MR2024}.

\begin{lemma}[Lemma 5.11, \cite{MR2024}]
\label{lemma_last}
The inequality \eqref{entropy_solutions_inequality} from Proposition \ref{proposition_entropy_solution_inequality} holds for all positive test functions $\psi \in \mathcal D^+(Q_{lagr})$ without restrictions on their supports (including $\zeta$-shocks).

\end{lemma}

For more details see the proof of \cite[Theorem 6.1]{MR2024}. It utilizes the result of Lemma~\ref{lemma_last} and the classical scheme of Kru\v{z}kov's theorem to prove that the difference between $\mathcal U$ and $\mathcal V$ cannot increase, thus solutions with the same initial-boundary conditions must coincide.

\section{Application to Riemann problem}
\label{sec:Riemann-problem}

In this Section we describe the solutions to the Riemann problem~\eqref{eq:main_system_chem_flood},~\eqref{eq:Riemann-problem} for 
a simple case of an S-shaped $f$ changing monotonicity with respect to $c$ exactly once. For a study of the sufficient conditions for an S-shaped flow function see \cite{Castaneda, RastS-Shaped}.
 
We impose, in addition to (F1)--(F2) from Section~\ref{sec:restrictions}, the further assumptions (F3)--(F4) on the fractional-flow function $f$:
\begin{enumerate}
    \item[(F3)] $f$ is $S$-shaped in $s$: for each $c \in [0,1]$ function $f(\cdot,c)$ has a unique point of inflection $s^I =  s^I(c) \in (0, 1)$, such that $f_{ss}(s, c)>0$ for $0<s<s^I$ and $f_{ss}(s, c)<0$ for $s^I<s<1$. 
    \item[(F4)] $f$ has exactly one change of monotonicity in $c$: $\exists c^*\in(0,1)$: 
    \begin{itemize}
        \item $f_c(s, c)<0$ for $0<s<1$, $0 < c < c^*$;
        \item $f_c(s, c)>0$ for $0<s<1$, $c^* < c < 1$.
    \end{itemize}
We assume that $f_{cc}(s^*,c^*)>0$, where $s^*\in[0,1]$ is a unique value that satisfies
     \begin{align*}
         f_s(s^*,c^*)=\frac{f(s^*,c^*)}{s^*+a_c(c^*)}.
     \end{align*}
\end{enumerate}
 
The main result of the Section states as follows.
\begin{theorem}
\label{thm2}
Consider Problem~\eqref{eq:main_system_chem_flood} under the assumptions (F1)--(F4) for the fractional flow function $f$ and (A1)--(A3) for adsorption function~$a$. Then for arbitrary states $u_L=(s_L,c_L)$, $u_R=(s_R,c_R)\in[0,1]^2$, $s_L\neq 0$ there exists a unique \mbox{$W$-solution} of the Riemann problem~\eqref{eq:Riemann-problem} (see Definition~\ref{def:solution}).
\end{theorem}
The existence will be proven by explicit construction, and the uniqueness is guaranteed by applying Theorem~\ref{thm:1}. For $s_L=0$ we constructed the solution (though it satisfies the vanishing viscosity criteria only when $c_L=c_R$), but the uniqueness is not formally covered by Theorem~\ref{thm:1} in its current form (see Remark~\ref{remark:c_L_0}). We hope to cover this case by some future generalization of the uniqueness theorem. First, in Section~\ref{subsec:conslaw-background} we recall important properties of hyperbolic conservation laws in general
and of the chemical flooding models in particular; then in Section~\ref{subsec:Riemann-solutions} we provide explicit solutions of the Riemann problem.

\subsection{Basic facts about conservation laws and chemical flooding models}
\label{subsec:conslaw-background}

In this section we recall the properties of hyperbolic conservation laws --- and, specifically, of the chemical flooding model~\eqref{eq:main_system_chem_flood} --- that are pertinent to the analysis of the Riemann problem.

\subsubsection{Characteristic speeds}
The system~\eqref{eq:main_system_chem_flood}
can be rewritten in the form
\begin{align*}
    u_t+A(u)u_x=0, \qquad u=(s,c),
\end{align*}
where $A(u)$ is the characteristic matrix
\begin{equation}
A(s,c) = 
\begin{pmatrix}
f_s & f_c \\
0 & f/(s+a_c(c))
\label{eq:characteristic-matrix}
\end{pmatrix}.
\end{equation}
For more details see~\cite[Section 2]{JnW}. The eigenvalues of $A(u)$,
\emph{i.e.},\ the characteristic speeds for the system,
are
\begin{equation}
\label{eq:lambda}
\lambda^s = f_s
\qquad \text{and} \qquad
\lambda^c = f/(s+a_c(c)).
\end{equation}
We choose the right eigenvectors
corresponding to $\lambda^s$ and $\lambda^c$ to be
\begin{equation}
\label{eq:eigenvectors}
r^s := \begin{pmatrix} 1 \\ 0 \end{pmatrix}
\qquad \text{and} \qquad
r^c := \begin{pmatrix} -f_c \\
\lambda^s - \lambda^c \end{pmatrix}.
\end{equation}
As both characteristic speeds are real, this model is hyperbolic.
However,
the characteristic speeds coincide,
not only on the boundary line $s = 0$,
but also the curve
\begin{align}
\label{eq:coincidence-alpha}
\mathcal{C}
:= \{\, (s,c)\in[0,1]^2 \,:\, \text{$\lambda^s(s,c) = \lambda^c(s,c)$ and $s \ne 0$} \,\},
\end{align}
thus the model is not strictly hyperbolic. The coincidence locus $\mathcal{C}$ divides the domain into two regions: left $\Omega_L$, where $\lambda^{s}>\lambda^{c}$, and right $\Omega_R$, where $\lambda^{s}<\lambda^{c}$.
\subsubsection{Rarefaction waves}

Let \(\lambda\) be an eigenvalue of the characteristic matrix given by~\eqref{eq:characteristic-matrix} with corresponding eigenvector~$r$.  
The simple rarefaction waves are continuous solutions of~\eqref{eq:main_system_chem_flood} of the form
\[
u(x,t) = v\!\left(\xi\right),\quad \xi=\frac{x}{t},
\]
where \(v\) corresponds to an integral curve of the vector field \(r\). More precisely,
\begin{equation}
\label{eq:rare}
u(x,t) =
\begin{cases}
u_L, & \text{if } x/t < \lambda(u_L), \\[0.3em]
v, & \text{if } x/t = \lambda(v), \\[0.3em]
u_R, & \text{if } x/t > \lambda(u_R),
\end{cases}
\end{equation}
where \(v\) is an integral curve of the vector field \(r\) connecting the states \(u_L\) and \(u_R\) with the additional property that the eigenvalue \(\lambda\) is increasing from \(u_L\) to \(u_R\).

Since the matrix \(A\) has two eigenvalues, \(\lambda^s\) and \(\lambda^c\), there are two possible rarefaction curves going through any given state \(u_L\):
\begin{itemize}
    \item  
If \(\lambda=\lambda^s\) with eigenvector \(r^s=(1,0)\), then \(c\) is constant along the integral curves. Thus, a simple rarefaction of the form~\eqref{eq:rare} exists whenever \(c_L=c_R\) and \(\lambda^s=f_s(s,c_L)\) increases from \(s_L\) to \(s_R\). This is precisely the Buckley--Leverett rarefaction for \(f(s)=f(s,c_L)\), hereafter called an \emph{s-rarefaction wave}.
\item If \(\lambda=\lambda^c\) with eigenvector \(r^c\) defined in~\eqref{eq:eigenvectors}, the integral curves are nontrivial in the \((s,c)\)-plane (see Fig.~\ref{fig:rarefactions}) and correspond to solutions of the following dynamical system (we call such curves \emph{$c$-rarefaction curves}):
\begin{align}
\label{eq:rare-dyn-sys}
    \begin{pmatrix}
    s\\
    c
    \end{pmatrix}_\xi=
    \begin{pmatrix} 
    -f_c \\ 
    \lambda^{s} - \lambda^{c}
    \end{pmatrix}.
\end{align}
A simple rarefaction of the form~\eqref{eq:rare} exists when \(\lambda^c\) increases along the curve connecting \(u_L\) and \(u_R\). We refer to such solutions as \emph{c-rarefaction waves}. 

\begin{remark}
\label{rmk:rare-s=0-s=1}
  We regard the constant states \(s \equiv 0\) and \(s \equiv 1\) as two \(c\)-rarefaction curves. Indeed, when $s_L=0$ or $s_L=1$, then $f_c(u_L)=0$, and the eigenvector $r^c$ can be taken to be $(0,1)$. For more details see discussion at the end of Section 3 in~\cite{JnW}.

\end{remark}
\begin{figure}[H]
    \centering
    \includegraphics[width=0.55\linewidth]{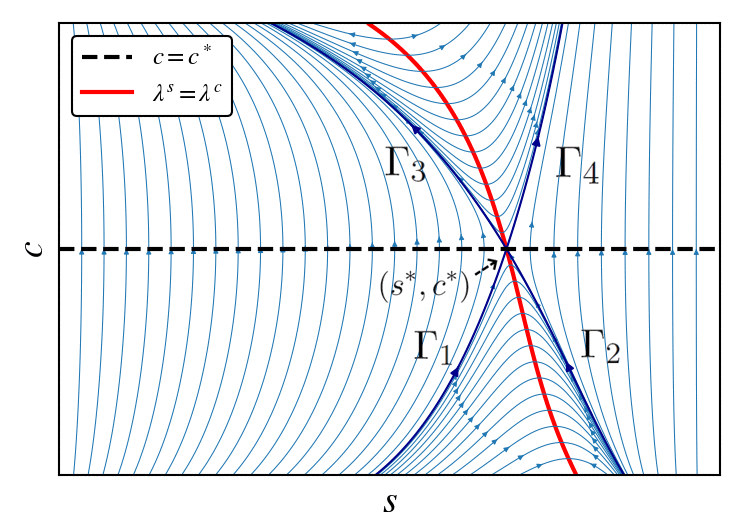}
    \caption{The blue curves represent the family of \(c\)-rarefaction curves --- integral curves of the system~\eqref{eq:rare-dyn-sys} associated with \(c\)-rarefaction waves. The critical point $(s^*,c^*)$ is of saddle-type. The arrows show the direction of increasing eigenvalue $\lambda^c$. The red curve represents the coincidence locus $\mathcal{C}$.}
    \label{fig:rarefactions}
\end{figure}
\end{itemize}
The following proposition collects the basic properties of $c$-rarefaction curves in $(s,c)$-plane that we will use to construct Riemann problem solutions.

\begin{proposition}
    \label{prop:rare}
    For the $c$-rarefaction curves the following properties are fulfilled:

    \begin{itemize}
        \item For any $(s_0,c_0)\in\Omega_L\cup\Omega_R$ the $c$-rarefaction curve $\Gamma$ that passes through the point $(s_0,c_0)$ can be written as a function $s=s(c)$ and
        \begin{align*}
            \frac{d}{dc}\lambda^{c}(s(c),c)> 0.
        \end{align*}
        \item The system~\eqref{eq:rare-dyn-sys} admits a unique fixed point of saddle type, $(s^*,c^*)$, defined by the assumption~(F4), with four $c$-rarefaction curves, 
        $\Gamma_i, i=1,\ldots,4,$ intersecting at
        this point. We adopt the following notation for the parametrizations of $\Gamma_i, i=1,\ldots,4,$ 
        (see Fig.~\ref{fig:rarefactions}):
         \begin{align*}
             s_1(c), c\in[0,c^*],\text{ such that }(s_1(c),c)\in\Omega_L\cup \mathcal{C};
             \\
             s_2(c), c\in[0,c^*],\text{ such that }(s_2(c),c)\in\Omega_R\cup \mathcal{C};
             \\
             s_3(c), c\in[c^*,1],\text{ such that }(s_3(c),c)\in\Omega_L\cup \mathcal{C};
             \\
             s_4(c), c\in[c^*,1],\text{ such that }(s_4(c),c)\in\Omega_R\cup \mathcal{C}.
         \end{align*}
    \end{itemize}

\end{proposition}
\begin{proof}
    The properties of rarefaction curves follow from the following easily verifiable formulas:
\begin{align*}
    (\lambda^c-\lambda^s)\cdot s'(c)=f_c,\qquad \frac{d}{dc}\lambda^{c}(s(c),c)=
    -\frac{a_{cc}(c)\cdot f(s(c),c)}{(s(c)+a_c(c))^2}.
\end{align*}
Due to assumptions (F1)--(F4), there exists a unique point $(s^*,c^*)$, which satisfies
\begin{align*}
    f_c(s^*,c^*)=0,\qquad \lambda^s(s^*,c^*)=\lambda^c(s^*,c^*),
\end{align*}
thus $(s^*,c^*)$ is a fixed point of the system~\eqref{eq:rare-dyn-sys}.
The linearization of the system~\eqref{eq:rare-dyn-sys} at the point $u^*=(s^*,c^*)$ states $u_\xi=L(u^*)u$ for the matrix $L$ defined as:
\begin{align*}
    L=\begin{pmatrix} 
    -f_{sc} &  -f_{cc}\\ 
    f_{ss} - \frac{\lambda^{s}-\lambda^{c}}{s+a_c} & 
    f_{sc} - \frac{f_c}{s+a_c} + \frac{f a_{cc}}{(s+a_c)^2}
    \end{pmatrix}.
\end{align*}
At fixed point $u^*$ the expression for $L$ simplifies:
\begin{align*}
    L(u^*)=\begin{pmatrix} 
    0 &  -f_{cc}\\ 
    f_{ss} & 
     \frac{f a_{cc}}{(s+a_c)^2}
    \end{pmatrix}.
\end{align*}
The equation on eigenvalues $\mu$ of $L(u^*)$ is
\begin{align*}
    \mu^2-\mu\cdot \frac{fa_{cc}}{(s+a_c)^2}+f_{ss}f_{cc}=0.
\end{align*}
The eigenvalues of $L(u^*)$ are
\begin{align*}
    \mu_{\pm}(u^*)=\frac{1}{2}\left(\frac{f a_{cc}}{(s+a_c)^2}\pm\sqrt{\mathcal{D}}\right),\qquad \mathcal{D}=\left(\frac{f a_{cc}}{(s+a_c)^2}\right)^2-4f_{ss}f_{cc}.
\end{align*}
Notice that at the point $u^*$ we have $f_{ss}<0$ and $f_{cc}>0$ due to assumptions (F1)--(F4). Therefore, $\mu_+(u^*)>0$ and $\mu_-(u^*)<0$, and  the point $u^*$ is a saddle point for the dynamical system~\eqref{eq:rare-dyn-sys}. See the qualitative picture of the set of the $c$-rarefaction curves in~Fig.~\ref{fig:rarefactions}.
\end{proof}
\begin{remark}
\label{rm:rare}
    Fix the state $(s_0,c_0)$ with $c_0<c^*$. If $s_0\in[0,s_1(c_0))\cup(s_2(c_0),1]$, the $c$-rarefaction curve that passes through the point $(s_0,c_0)$ is defined for all $c\in[0,1]$; meanwhile if $s_0\in(s_1(c_0),s_2(c_0))$ this rarefaction curve  reaches the coincidence locus at some point $(\tilde{s},\tilde{c})\in\mathcal{C}$ and can not be defined for $c>\tilde{c}$. For more details see~\cite[Section 3]{JnW}. In particular, this means that when $s_0\not=s_1(c_0)$ and $s_0\not=s_2(c_0)$, the $c$-rarefaction curve lies either in $\Omega_L\cup\,\mathcal{C}$ or in $\Omega_R\cup\,\mathcal{C}$. On the other hand, if $s_0=s_2(c_0)$, the union $\Gamma_2\cup\Gamma_3$ is also an integral curve of the system~\eqref{eq:rare-dyn-sys}, and can be seen as a unique $c$-rarefaction curve (similarly, $\Gamma_1\cup\Gamma_4$). This is one of the major differences with the monotone case considered in~\cite{JnW}.
    
\end{remark}

\subsubsection{Shock waves}
Recall the standard ordering of the eigenvalues
of the characteristic matrix~\eqref{eq:characteristic-matrix}
as $\lambda_1(u)<\lambda_2(u)$,
called the 1-family and 2-family characteristic speeds.
In polymer models,
$\lambda_1(u)$ equals $\lambda^c(u)$
when $u\in\Omega_L$,
but equals $\lambda^s(u)$
when $u\in\Omega_R$.

A standard way to classify a discontinuity
is based on the ordering of the characteristic speeds
on its two sides relative to its propagation speed $v$,
\emph{i.e.},\ $\lambda_i(u_-) - v$ and $\lambda_i(u_+) - v$ for $i = 1, 2$ (see, for example,~\cite{Lax57}, \cite{IsaMarPlohr}, and~\cite[Chapter 8]{Dafermos}).
Four of the possibilities, which we call
the 1-family Lax, 2-family Lax, overcompressive, and crossing
configurations of characteristic paths,
are depicted in Fig.~\ref{fig:types_shocks}:
\begin{itemize}
\item 1-family Lax: $\lambda_1(u_-)>v>\lambda_1(u_+)$,
$v <\lambda_2(u_-)$, and $v<\lambda_2(u_+)$;
\item 2-family Lax: $\lambda_2(u_-)>v>\lambda_2(u_+)$,
$v >\lambda_1(u_-)$, and $v>\lambda_1(u_+)$;
\item overcompressive: $\lambda_1(u_-)>v>\lambda_1(u_+)$
and $\lambda_2(u_-)>v>\lambda_2(u_+)$;
\item crossing:
$\lambda_2(u_-)>v>\lambda_1(u_-)$
and $\lambda_1(u_+)<v<\lambda_2(u_+)$.
\end{itemize}

\begin{figure}[ht]
    \centering
    \includegraphics[width=0.24\textwidth]{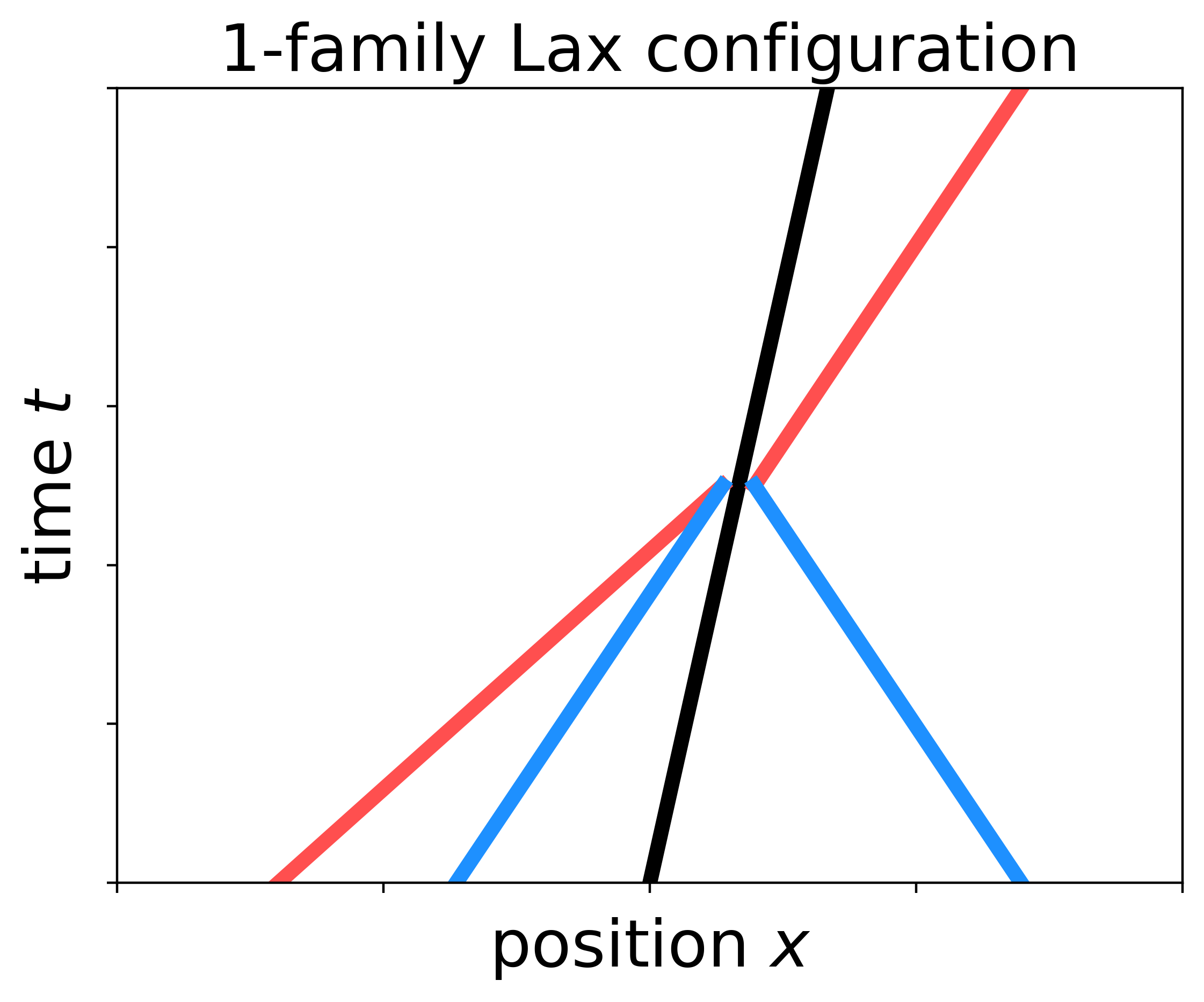}
    \includegraphics[width=0.24\textwidth]{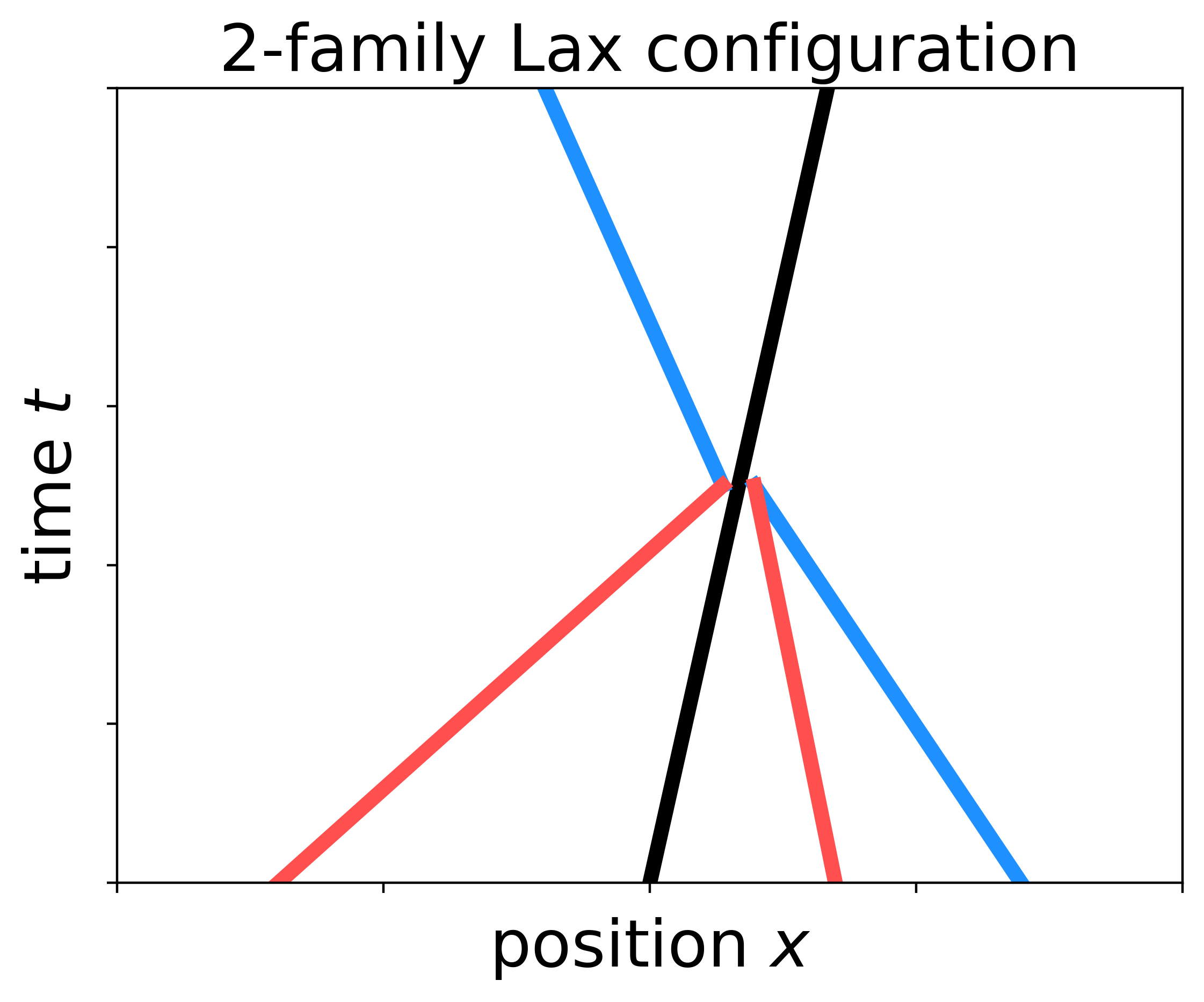}
    \includegraphics[width=0.24\textwidth]{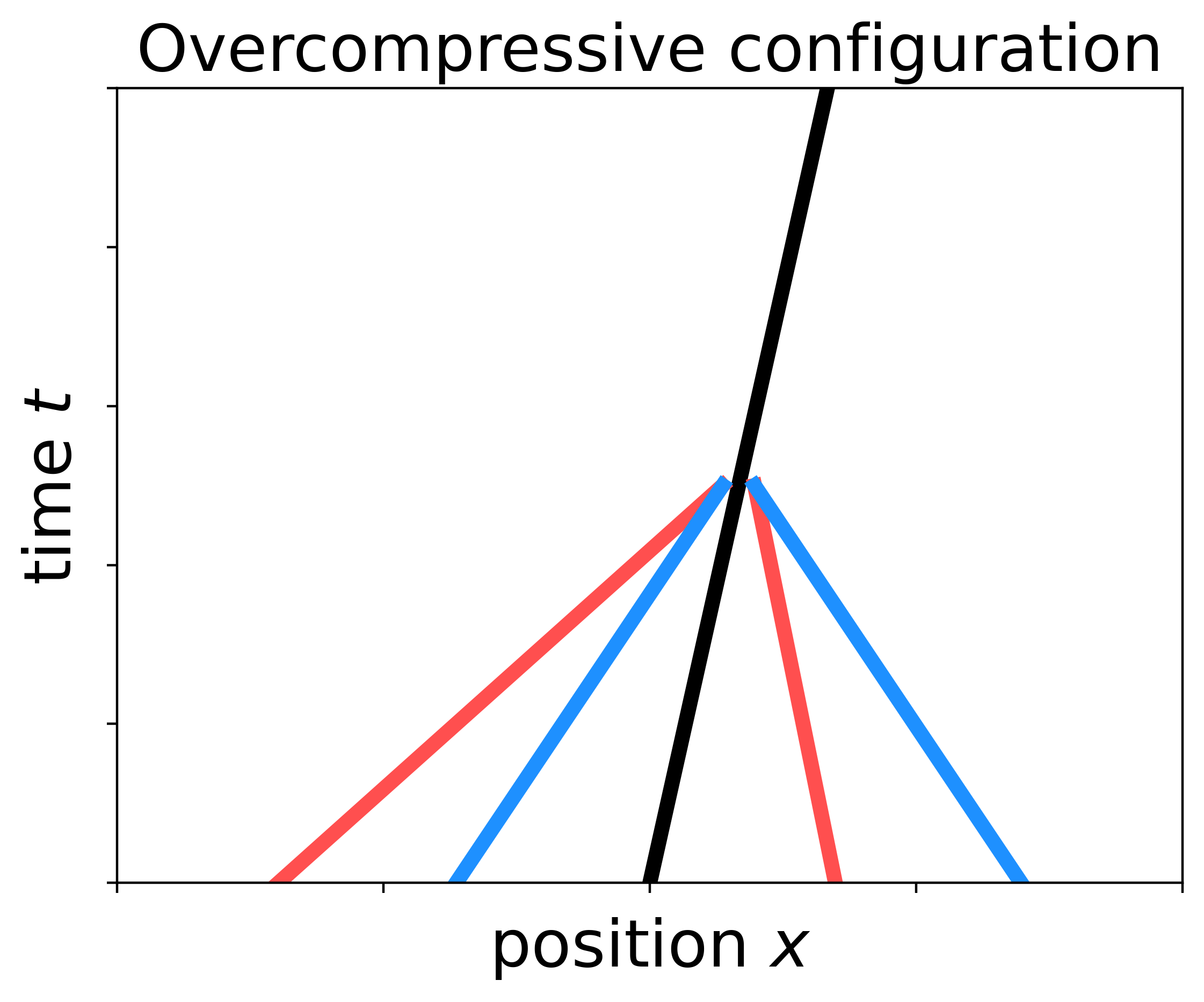}
    \includegraphics[width=0.24\textwidth]{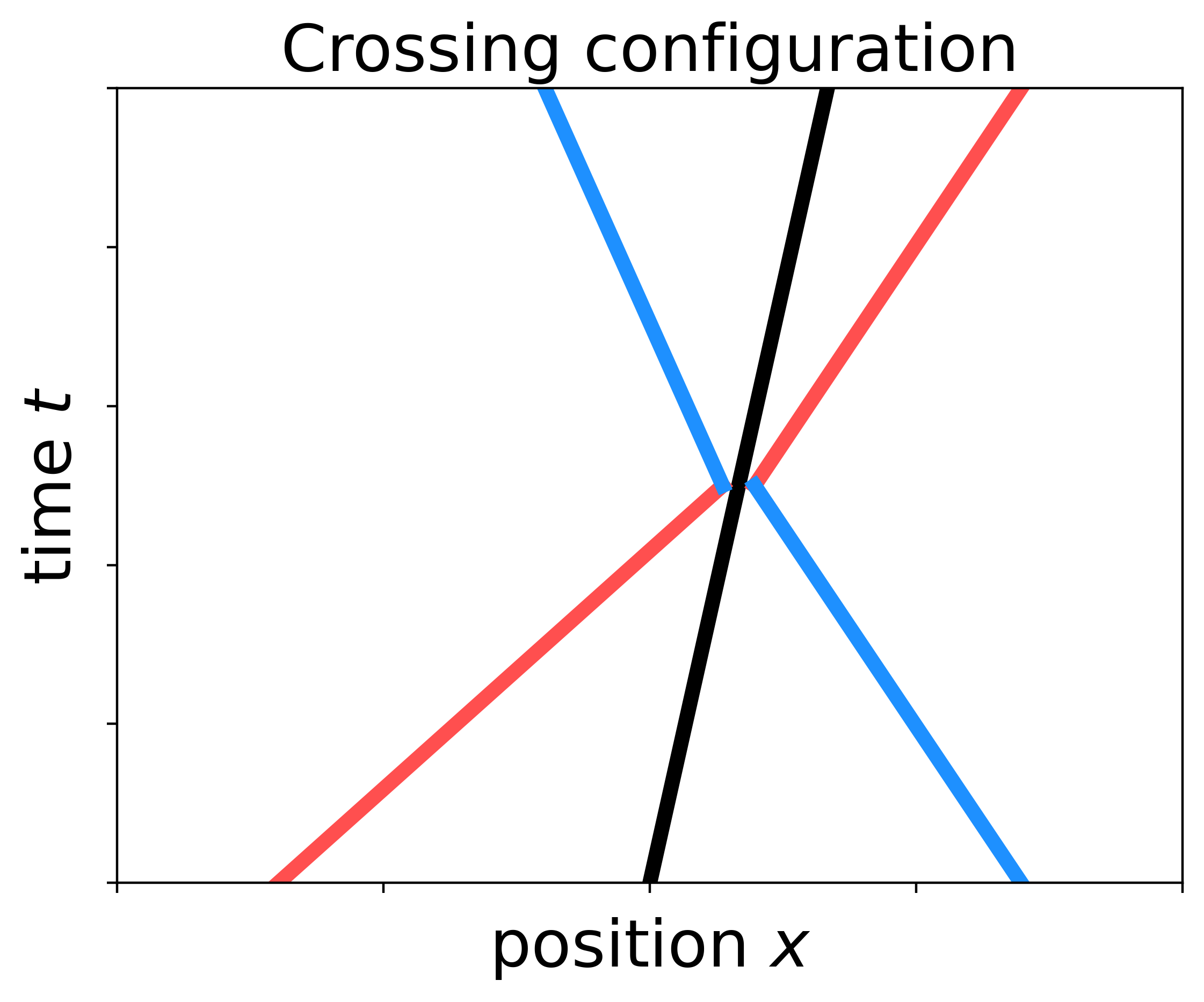}
    \caption{Four possible configurations of characteristic paths
    drawn in $(x,t)$-space: 1-family Lax, 2-family Lax, overcompressive,
    and crossing. The black line is the shock trajectory with
    speed $v$. Blue curves are characteristic paths for the 1-family,
    and red curves are for the 2-family.}
    \label{fig:types_shocks}
\end{figure}

\subsubsection{Compatibility by speed}
A solution of a Riemann problem is an assembly
of $s$-wave groups, $c$-waves,
and constant states ordered by speed. We shall use the notation $u\xrightarrow{s}u'$
(respectively, $u\xrightarrow{c} u'$)
to denote an $s$-wave group (resp., a $c$-wave) connecting states $u$ and $u'$
in the direction of increasing speed. We say that the two waves $u\xrightarrow{\phantom{s}}u'$ and $u'\xrightarrow{\phantom{s}}u''$ are \emph{compatible by speed} if the speed increases when we go through the state~$u'$, and therefore they can be composed to solve the Riemann problem with left state $u$ and right state $u''$.

As noted in Lemma~5.1 of~\cite{JnW}, the solution $u=(s,c)$ of a Riemann problem has the property that the function $c(x,t)$ is a monotone function of $x$ for every $t\geq0$. In the proof the authors do not assume the monotonicity of $f$ in $c$ and use only the concavity of $a$, thus this lemma is also valid under our assumptions on $f$ and $a$. For the reader's convenience, we recall this lemma below.
\begin{proposition}[Lemma~5.1,~\cite{JnW}]
    \label{prop:JW-monotonicity-in-c}
    Assume that the three waves
    \begin{align*}
        u_L\xrightarrow{c\text{-wave}} u_1\xrightarrow{s\text{-wave}} u_2\xrightarrow{c\text{-wave}} u_R
    \end{align*}
    are compatible by speed. Then both $c$-waves are rarefaction waves.
\end{proposition}
Let \(u_L=(s_L,c_L)\) and \(u_R=(s_R,c_R)\) denote the left and right states of the Riemann problem, respectively. By Proposition~\ref{prop:JW-monotonicity-in-c}, if \(c_L<c_R\), any solution of the Riemann problem consists of \(s\)-waves together with \(c\)-rarefaction waves; if \(c_L>c_R\), any solution consists of \(s\)-waves and a single \(c\)-shock. Moreover, if \(c_L=c_R\), any solution reduces to a single \(s\)-wave. Hence, by the theory of the Buckley--Leverett equation~\cite{BL}, the Riemann problem for the system~\eqref{eq:main_system_chem_flood} admits a unique solution when \(c_L=c_R\). Therefore, it remains to prove Theorem~\ref{thm2} in the case \(c_L\neq c_R\).

\subsection{Region layout for the Riemann problem solutions}
\label{subsec:Riemann-solutions}
In this section we describe the layout of the regions of $(u_L, u_R)$ in which the Riemann problem solutions have similar structure (the same sequence of $s$- and $c$-waves).
In Section~\ref{subsec:RP-1-c-shock} we treat the case $c_L>c_R$, while the case $c_L<c_R$ is considered in Section~\ref{subsec:RP-rare}.

\subsubsection{Case $c_L>c_R$. }
\label{subsec:RP-1-c-shock}
Throughout this section we shall  assume that the values of $c_L$ and $c_R$ are fixed with $c_L>c_R$. Our aim is to prove that for any $s_L, s_R\in[0,1]$ there exists a unique solution to a Riemann problem~\eqref{eq:main_system_chem_flood},~\eqref{eq:Riemann-problem}. Recall that due to Proposition~\ref{prop:JW-monotonicity-in-c} the case $c_L>c_R$ corresponds to the solution of a Riemann problem with at most one $c$-shock,~i.e. the solution has the following structure in terms of $s$ and $c$-waves:
\begin{align}
\label{eq:scs-shock}
    u_L\xrightarrow{s\text{-wave}} u^{-} \xrightarrow{c\text{-shock}} u^{+} \xrightarrow{s\text{-wave}} u_R,
\end{align}
where the first, the last or both $s$-waves may be absent. The solutions of the form~\eqref{eq:scs-shock} clearly satisfy the conditions~(W1)--(W3) from the definition~\eqref{def:solution} of W-solution. Moreover, as described below, the $c$-shock wave is obtained as a limit of travelling wave solutions, thus condition~(W4) is also fulfilled.

The case $c^*\geq c_L>c_R$ corresponds to the monotone dependence $f$ on $c$ (decreasing in $c$) and the explicit solutions to the Riemann problem were constructed in~\cite[Section~7]{JnW}. Note that in this case the set of admissible shock waves does not depend on the choice of $\kappa=\varepsilon_d/\varepsilon_c$ (see~\eqref{eq:main_system_dissipative}), thus the set of admissible Riemann solutions in this paper and in~\cite{JnW} is the same. The case $c_L>c_R\geq c^*$ also corresponds to the monotone dependence $f$ on $c$ (increasing in $c$) and can be solved using the same approach. Therefore, it is enough to consider the case $c_L>c^*>c_R$.

Consider a $c$-shock wave between the states $u^-=(s^-,c_L)$ and $u^+=(s^+,c_R)$ with the velocity $v$, defined by the Rankine-Hugoniot conditions~\eqref{eq:RH-1}, i.e.
\begin{align}
\label{eq:RH-speed}
    v = \dfrac{[f(s,c)]}{[s]} = \dfrac{f(s^+,c_R)}{s^+ + h}, \qquad h = \dfrac{[a(c)]}{[c]},
\end{align}
where $[q(s,c)]=q(s^+,c_R)-q(s^-,c_L)$ in this case. It will be convenient to have the following notation (as in~\cite[Section~7]{JnW}). 
    Due to condition (F3) on the flow function~$f$, it is clear that for the state $u^+=(s^+, c_R)$ 
    there exists at most one value $s^K=s^K(u^+)\not=s^+$ such that
\begin{align*}
     \dfrac{f(s^K,c_R)}{s^K + h}=v.
\end{align*}
If such $s^K$ exists, we call it the \textit{critical shock value}. Geometrically, $s^K$ corresponds to the abscissa of the intersection point of the graph of $f(s,c_R)$ and the line that passes through the point $(s^+, f(u^+))$ with inclination equal to $v$ (in particular, the line also passes through the points $(s^-,f(u^-))$ and $(-h,0)$, see Fig.~\ref{fig:sK}).

\begin{figure}[h]
    \centering
    \includegraphics[width=0.6\linewidth]{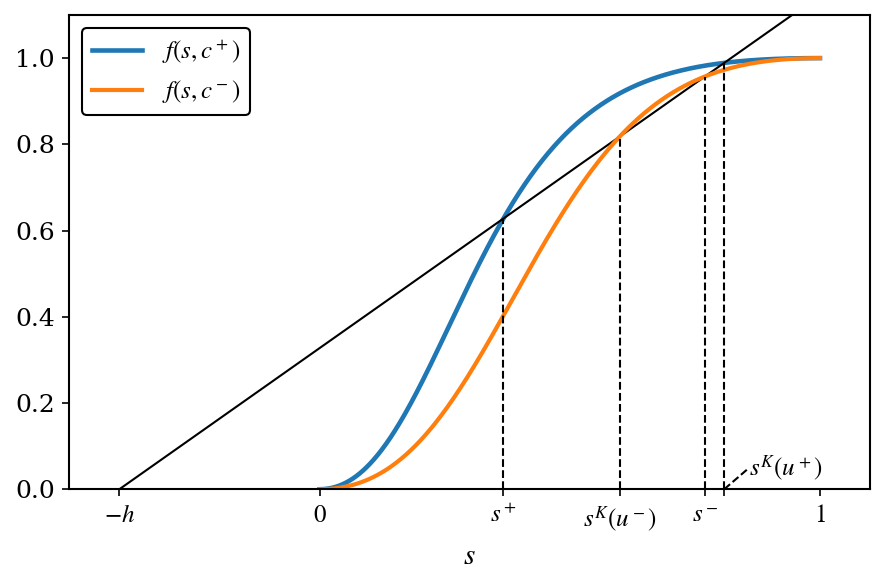}
    \caption{Geometrical interpretation of the critical values $s^K(u^+)$ and $s^K(u^-)$.}
    \label{fig:sK}
\end{figure}

Let us first formulate the extension of Theorem~3.2 in~\cite{Bahetal} for a larger class of states $u_L$ and $u_R$ (in~\cite{Bahetal} $c_L=1$, $c_R=0$, but the proof is identical for any $c_L>c^*>c_R$). Lemma~\ref{lm:cLcR} serves as a vanishing viscosity criterion for choosing the admissible $c$-shock wave with crossing configuration (such waves are also called undercompressive or transitional shock waves), see condition (W4). 
\begin{lemma}
\label{lm:cLcR}
Consider a system of conservation laws~\eqref{eq:main_system_chem_flood} under assumptions (F1)--(F4) and (A1)--(A3). Fix $c_L$ and $c_R$ such that $c_L>c^*>c_R$. Also fix the parameters $\varepsilon_c, \varepsilon_d>0$ of the equation~\eqref{eq:main_system_dissipative}. Then there exist $0<v_{\min}<v_{\max}<\infty$, such that  for every $\kappa=\varepsilon_d/\varepsilon_c\in(0, +\infty)$, there exist unique 
\begin{itemize}
    \item points $s^-:=s^-(\kappa;c_L,c_R)\in[0,1]$ and $s^+:=s^+(\kappa;c_L,c_R)\in[0,1]$;
    \item velocity $v:=v(\kappa;c_L,c_R)\in[v_{\min},v_{\max}]$,
\end{itemize}
such that the $c$-shock wave, connecting $u^-:=u^-(\kappa;c_L,c_R)=(s^-,c_L)$ and $u^+:=u^+(\kappa;c_L,c_R)=(s^+,c_R)$ with velocity~$v$, is admissible by vanishing viscosity criterion (see condition (W4)) and is  of crossing configuration.

Furthermore, the following sequence of waves is compatible by speeds
\begin{align*}
    u_L=(s_L,c_L)\xrightarrow{s} u^{-} \xrightarrow{c\text{-shock}} u^{+} \xrightarrow{s} u_R=(s_R,c_R)
\end{align*}
for $c_L>c^*>c_R$ and all $s_L$ and $s_R$ under the condition
\begin{align}
\label{eq:cond-sL-sR}
    s^K(u^-)\leq s_L\leq 1\qquad\text{and}\qquad 0\leq s_R\leq s^K(u^+).
\end{align}
\end{lemma}

Finally, let us prove the Theorem~\ref{thm2} for the case $c_L>c^*>c_R$, $c_L$ and $c_R$ fixed. Consider $s^-$, $s^+$, $u^-$, $u^+$ and $v$ from  Lemma~\ref{lm:cLcR}. 

\begin{figure}[h]
    \centering
    \includegraphics[width=0.5\linewidth]{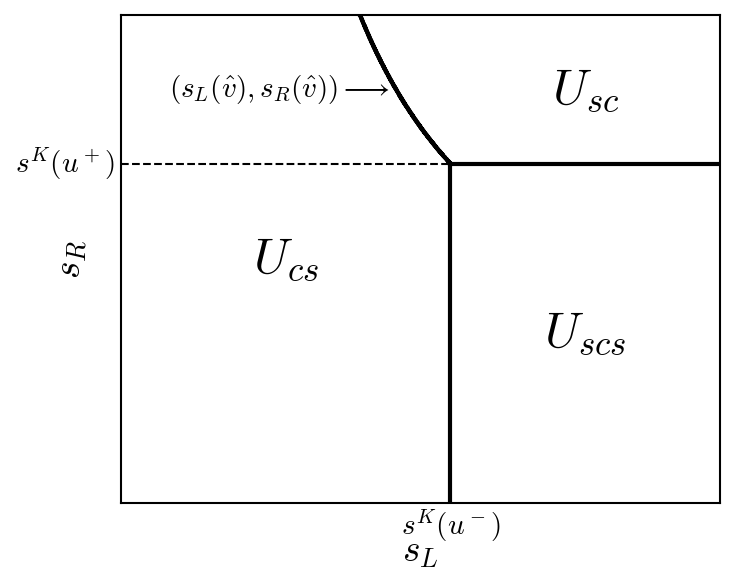}
       
    \caption{Subdivision into three regions in the $(s_L,s_R)$-plane with different structure of solutions in terms of the sequence of $s$-wave groups and $c$-waves:  $\mathbf{U}_{cs}, \mathbf{U}_{sc}$ and $\mathbf{U}_{scs}$.}
    \label{fig:RP-layouts-shocks-1}
\end{figure}
There exist four possible layouts for the structure of the solution to the Riemann problem as a sequence of $s$-wave groups and $c$-waves depending on  $(s_L,s_R)\in[0,1]^2$, see Fig.~\ref{fig:RP-layouts-shocks-1}:

\begin{enumerate}
    \item \textbf{Undercompressive shock.}
    
    Consider a region:
    \begin{align*}
        \mathbf{U}_{scs}=[s^K(u^-),1]\times[0,s^K(u^+)].
    \end{align*}
    If $(s_L,s_R)\in\mathbf{U}_{scs}$, then by Lemma~\ref{lm:cLcR} the following sequence of waves $(scs)$ provides a solution to the Riemann problem:
    \begin{align}
\label{eq:solution-RP-1}
    u_L=(s_L,c_L)\xrightarrow{s} u^{-} \xrightarrow{c\text{-shock}} u^{+} \xrightarrow{s} u_R=(s_R,c_R).
\end{align}
As mentioned in Lemma~\ref{lm:cLcR}, this  $c$-shock wave is undercompressive. If $s_L=s^-$, then the first $s$-wave in~\eqref{eq:solution-RP-1} is absent. Similarly, if $s_R=s^+$, then the last $s$-wave in~\eqref{eq:solution-RP-1} is absent.
\begin{figure}[h]
    \centering
\includegraphics[width=0.32\linewidth]{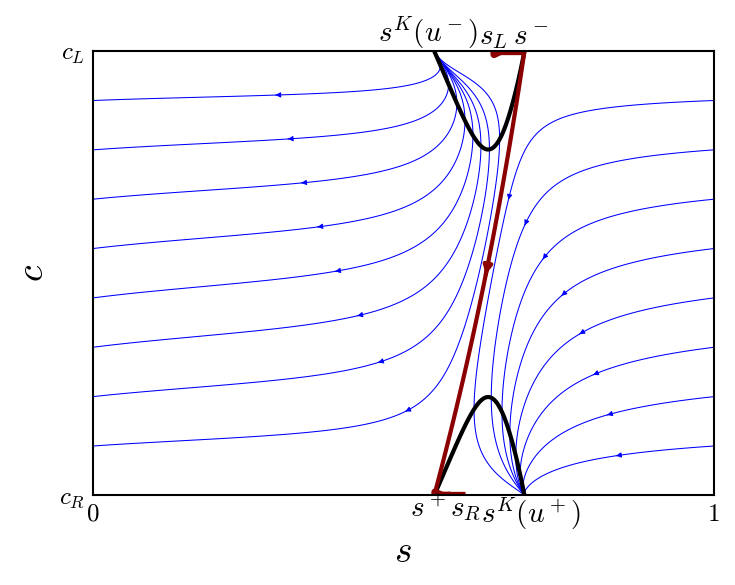}
    \includegraphics[width=0.32\linewidth]{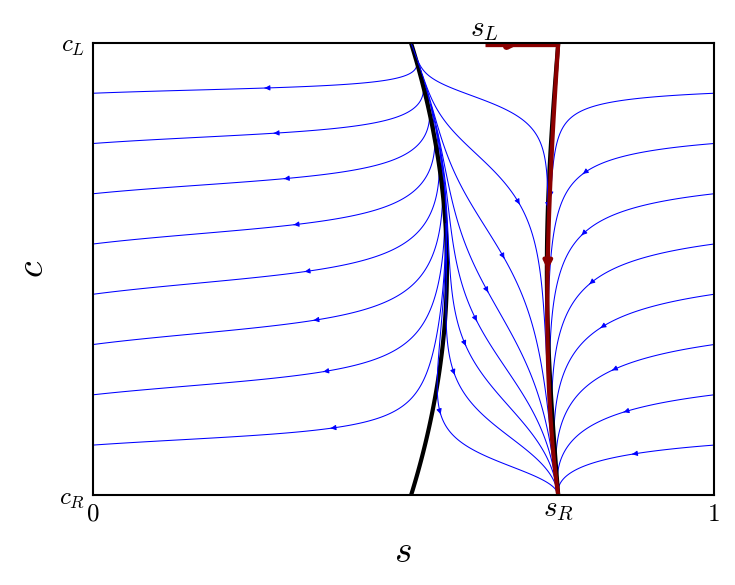}
    \includegraphics[width=0.32\linewidth]{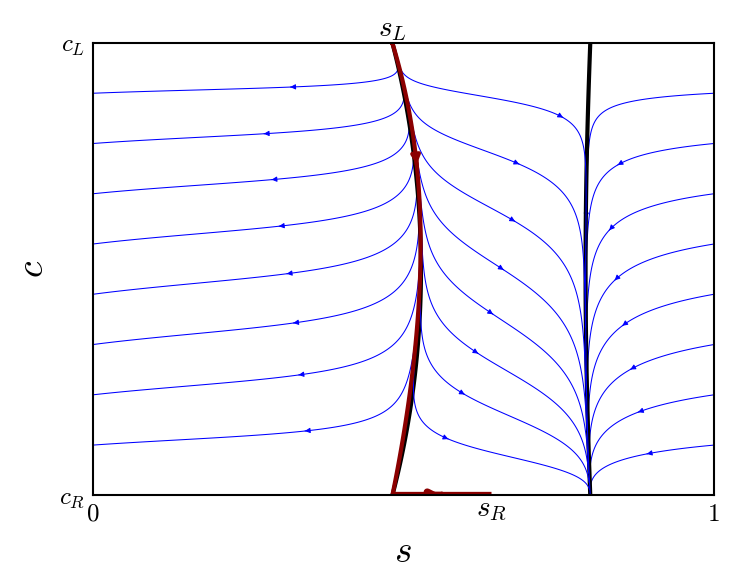}
    \\
 \qquad $(s_L,s_R)\in\mathbf{U}_{scs}$
 \qquad \qquad\qquad
 $(s_L,s_R)\in\mathbf{U}_{sc}$
 \qquad \qquad \qquad
 $(s_L,s_R)\in\mathbf{U}_{cs}$
       \caption{
       The dark red curves represent the sequence of waves~\eqref{eq:solution-RP-1},~\eqref{eq:sc-fast} and~\eqref{eq:cs-slow} in $(s,c)$-plane for three cases: $(s_L,s_R)\in\mathbf{U}_{scs}$, $(s_L,s_R)\in\mathbf{U}_{sc}$ and $(s_L,s_R)\in\mathbf{U}_{cs}$.
       The light blue curves illustrate the trajectories of the corresponding dynamic system~\eqref{eq:dyn_sys_cap_diff} and the dark blue curves illustrate the nullclines, see Section~\ref{sec:sec2-Hopf}.}
    \label{fig:RP-layouts-shocks}
\end{figure}

    \item \textbf{Overcompressive shock.}
    
    Let $v_1\in\mathbb{R}$ be defined as:
\begin{align*}
    v_1=\frac{1}{1+h}, \quad h=\frac{a(c_L)-a(c_R)}{c_L-c_R}.
\end{align*}

    For any speed $\hat{v}\in[v_1,v]$ there exist four critical points of the travelling wave dinamic system \eqref{eq:dyn_sys_cap_diff}: $(s_L^{1,2}, c_L)$, $(s_R^{1,2}, c^R)$. Any two of these points satisfy the Rankine-Hugoniot condition \eqref{eq:RH-speed}. If we denote $s_L(\hat{v}) = s_L^1 < s_L^2$ and $s_R(\hat{v}) = s_R^2 > s_R^1$, then these saturation values have the following properties
    \begin{enumerate}
        \item $s_L(\hat{v})\in[0,s^K(u^-)]$; \quad $s_R(\hat{v})\in[s^K(u^+),1]$;
        \item the states $u_L=(s_L(\hat{v}),c_L)$, $u_R=(s_R(\hat{v}),c_R)$ and the speed $\hat{v}$ satisfy the Rankine-Hugoniot conditions~\eqref{eq:RH-speed};
        \item the following sequence of waves $(cs)$ provides a solution to the Riemann problem:
        \begin{align}
        \label{eq:cs-overcompressive}
    u_L=(s_L(\hat{v}),c_L)\xrightarrow{c\text{-shock}} \hat{u}=(s^K(u_R),c_R) \xrightarrow{s\text{-shock}} u_R=(s_R(\hat{v}),c_R).
\end{align}  
Note that the speeds of $c$-shock and $s$-shock are both equal to $\hat{v}$.
    \end{enumerate}
    
Also there is an alternative way to represent the solution~\eqref{eq:cs-overcompressive} as a combination of $s$-shock and $c$-shock  with both speeds equal to $\hat{v}$ (the order is different, $(sc)$):
    \begin{align}
    u_L=(s_L(\hat{v}),c_L)\xrightarrow{s\text{-shock}} \hat{u}=(s^K(u_L),c_L) \xrightarrow{c\text{-shock}} u_R=(s_R(\hat{v}),c_R).
\end{align}    
    
\begin{remark}
\label{rmk:overcompressive-curve}
    The set of all states $(s_L(\hat{v}),s_R(\hat{v}))$ in the plane $(s_L,s_R)$, described above, forms a curve parametrized by the speed $\hat{v}\in[v_1,v]$. This curve serves as a boundary between the two regions $\mathbf{U}_{sc}$ and $\mathbf{U}_{cs}$, where the structure of the solutions is of the type $(sc)$ and $(cs)$, respectively (see Fig.~\ref{fig:RP-layouts-shocks-1}). As the dependence of $s_R(\hat{v})$ is strictly monotone (by construction), we can define an inverse function, $\hat{v}(s_R)$. Hereinafter, we will consider a function $\hat{s}_L(s_R)=s_L(\hat{v}(s_R))$ defined for~$s_R\in[s^K(u^+),1]$.
\end{remark}
 
    Notice that the solution~\eqref{eq:cs-overcompressive} can be seen formally as one $c$-shock wave, and following the nomenclature in Section~\ref{subsec:conslaw-background}, it corresponds to an overcompressive shock.   
    \item \textbf{2-family Lax shock, fast.}
    
    Take $\hat{s}_L(s_R)$ from Remark~\ref{rmk:overcompressive-curve} and consider the region
    \begin{align*}
    \mathbf{U}_{sc}:=\{(s_L,s_R)\in[0,1]^2: s_R\in [s^K(u^+),1],\;s_L\in[\hat{s}_L(s_R), 1]\}.
    \end{align*}
    For any $(s_L,s_R)\in\mathbf{U}_{sc}$ the following sequence of waves $(sc)$ provides a solution to the Riemann problem:
     \begin{align}
  \label{eq:sc-fast}
    u_L=(s_L,c_L)\xrightarrow{s\text{-wave}} u_M=(s_M,c_L) \xrightarrow{c\text{-shock}} u_R=(s_R,c_R),
\end{align}   
where $u_M=(s_M,c_L)$ is uniquely defined, as it satisfies:
\begin{enumerate}
    \item  Rankine-Hugoniot condition~\eqref{eq:RH-speed} for $u_M$, $u_R$ and the corresponding speed~$v_M$;
    \item $\lambda^s(u_M)<v_M$.
\end{enumerate}
Notice that the $c$-shock from~\eqref{eq:sc-fast} is a 2-family Lax shock.
    \item \textbf{1-family Lax shock, slow.}
    
    Consider the region
    \begin{align*}
    \mathbf{U}_{cs}:=\textrm{closure}([0,1]^2\setminus(\mathbf{U}_{scs}\cup \mathbf{U}_{sc}))\setminus\{s_L=0\}.
    \end{align*}
    For any $(s_L,c_L)\in\mathbf{U}_{cs}$ the following sequence of waves $(cs)$ provides a solution to the Riemann problem:
        \begin{align}
        \label{eq:cs-slow}
    u_L=(s_L,c_L)\xrightarrow{c\text{-shock}} u_M=(s_M,c_R) \xrightarrow{s\text{-wave}} u_R=(s_R,c_R),
\end{align}  
where $u_M=(s_M,c_L)$ is uniquely defined, as it satisfies:
\begin{enumerate}
    \item  Rankine-Hugoniot condition~\eqref{eq:RH-speed} for $u_L$, $u_M$ and corresponding speed~$v_M$;
    \item $\lambda^s(u_M)>v_M$.
\end{enumerate}
Notice that the $c$-shock from~\eqref{eq:cs-slow} is 1-family Lax shock. Moreover, note that if $s_L=0$, the $c$-shock is not (W4)-admissible. It is also physically meaningless to assign chemical concentration when water saturation is zero. 
\end{enumerate}

\subsubsection{Case $c_L<c_R$. }
\label{subsec:RP-rare}
Throughout this section we shall  assume that the values of $c_L$ and $c_R$ are fixed with $c_L<c_R$. Our aim is to prove that for any $s_L, s_R\in[0,1]$ there exists a unique solution to a Riemann problem~\eqref{eq:main_system_chem_flood},~\eqref{eq:Riemann-problem}. Recall that due to Proposition~\ref{prop:JW-monotonicity-in-c}, the solution of a Riemann problem for the case $c_L<c_R$  corresponds to 
a combination of $s$-waves and $c$-rarefaction waves
\begin{align*}
    u_L\xrightarrow{s\text{-wave}} u_1 \xrightarrow{c\text{-rarefaction}} u_2\xrightarrow{s\text{-wave}} \ldots\xrightarrow{c\text{-rarefaction}} u_{k} \xrightarrow{s\text{-wave}} u_R
\end{align*}
for some $k\in\mathbb{N}$. Similarly to the previous case, $c^*\geq c_R>c_L$ and $c^*\leq c_L<c_R$ correspond to the monotone dependence of the fractional flow function $f(s,c)$ on $c$ (see~\cite[Section~6]{JnW} for a full description of the solutions of the Riemann problem when $c^*\geq c_R>c_L$). Therefore, it is sufficient to consider the case $c_L<c^*<c_R$.

\paragraph{Critical value $s_{\mathcal{K}}(u)$ and its properties.}

Following~\cite[Section~6]{JnW}, we adopt the following notation. For the state $u=(s, c)\in[0,1]^2$,  there exists at most one value $s_{\mathcal{K}}(u)\not=s$ such that
\begin{align}
\label{eq:sK-rare}
     \lambda^c(s_{\mathcal{K}}(u),c)=\lambda^c(s,c).
\end{align}
If such $s_{\mathcal{K}}(u)$ exists, we call it the \textit{critical rarefaction value}. 
Geometrically, $s_{\mathcal{K}}(u)$ corresponds to the abscissa of the intersection point of the graph of $f(\cdot,c)$ and the line connecting $(s, f(s, c))$ and $(-a_c(c),0)$. The corresponding state $u_{\mathcal{K}}(u)=(s_{\mathcal{K}}(u),c)$ is called the critical rarefaction state. If $(s,c)\in\mathcal{C}$, it will be convenient to consider $s_{\mathcal{K}}(s,c)=s$. If $\Gamma$ is a $c$-rarefaction curve, then the curve $\Gamma_{\mathcal{K}}$, which consists of all critical states for the $c$-rarefaction curve $\Gamma$, will be called the critical curve.

The following proposition explains why we call $s_{\mathcal{K}}(u)$ critical (this is a trivial generalization of~\cite[Lemma 6.1]{JnW} for $f$ under the conditions (F1)--(F4)).
\begin{proposition}
\label{prop:critical-value-properties}
The following is true:
\begin{itemize}
    \item The two waves 
    \begin{align*}
    u_L\xrightarrow{c} u_M\xrightarrow{s} u_R
    \end{align*}
    are compatible if and only if $u_M\in\Omega_L\cup \mathcal{C}$ and $s_R\in[0,s_{\mathcal{K}}(u_M)]$.
    \item The two waves 
    \begin{align*}
    u_L\xrightarrow{s} u_M\xrightarrow{c} u_R
    \end{align*}
    are compatible if and only if $u_M\in\Omega_R\cup \mathcal{C}$ and $s_L\in[s_{\mathcal{K}}(u_M),1]$.
    \end{itemize}
    \end{proposition}
\begin{remark}
    For the monotone case considered in~\cite{JnW}, Proposition~\ref{prop:critical-value-properties} implies that any solution to a Riemann problem contains at most two $c$-rarefaction waves. Note that this is not true anymore for the non-monotone case under the assumption (F4). Indeed, there exists a solution that is composed of three rarefaction waves (for more details, see Fig.~\ref{fig:RP-layouts-rare-5}c).    
\end{remark}

The Proposition~\ref{prop:sK-basic-properties} follows immediately from the assumptions~(F1)--(F3) for the fractional flow function $f$ and the definition of the critical value $s_{\mathcal{K}}(u)$.
\begin{proposition} 
\label{prop:sK-basic-properties}
Fix the state $u=(s,c)\in\Omega_L\cup\Omega_R$. The following properties hold:
    \begin{itemize}
        \item if $u\in\Omega_L$ and $s_{\mathcal{K}}(u)$ exist, then $u_{\mathcal{K}}=(s_{\mathcal{K}}(u),c)\in\Omega_R$. Moreover, $s_{\mathcal{K}}(u)>s$, and the mapping $s\mapsto s_{\mathcal{K}}(u)$ is continuous and strictly decreasing.
        
        \item if $u\in\Omega_R$, then $s_{\mathcal{K}}(u)$ exists and $u_{\mathcal{K}}=(s_{\mathcal{K}}(u),c)\in\Omega_L$. Moreover, $s_{\mathcal{K}}(u)<s$, and the mapping $s\mapsto s_{\mathcal{K}}(u)$ is continuous and strictly decreasing.
        
        \item if $s_{\mathcal{K}}(u)$ exists, then $u_{\mathcal{KK}}=(s_{\mathcal{K}}(u_{\mathcal{K}}),c)$ also exists, and $u_{\mathcal{KK}}=u$.
    \end{itemize}
\end{proposition}

\begin{proposition}
\label{prop:crit}        
    Fix a point $u_0 = (s_0,c_0)\in\Omega_L\cup\Omega_R$, such that $s_0, s_{\mathcal{K}}(u_0)<1$ exists, and consider  three curves (here $\delta>0$ is some small number such that the curves are well-defined in the $\delta$-neighborhood):
    \begin{itemize}
        \item the $c$-rarefaction curve $\Gamma$ (parametrized as in Proposition~\ref{prop:rare}) that passes through the point  $(s_0,c_0)$:
\begin{align*}
\Gamma=\{u(c)=(s(c),c):  c\in[c_0-\delta,c_0+\delta] \text{ and }s(c_0)=s_0\};
\end{align*}
\item the critical curve $\Gamma_{\mathcal{K}}$ consisting of all critical states for the $c$-rarefaction curve~$\Gamma$:
\begin{align*}
\Gamma_{\mathcal{K}}=\{(s_{\mathcal{K}}(c),c): c\in[c_0-\delta,c_0+\delta],
\text{ where }s_{\mathcal{K}}(c)= s_{\mathcal{K}}(u(c)), u(c)\in\Gamma\};
\end{align*}
\item  the $c$-rarefaction curve that passes through the critical state  $u_{\mathcal{K}}(u_0)\in\Gamma_{\mathcal{K}}$.
\begin{align*}
\widetilde{\Gamma}=\{(\tilde{s}(c),c): 
c\in[c_0-\delta,c_0+\delta]
\text{ and }\tilde{s}(c_0)=s_{\mathcal{K}}(u_0)\}.    
\end{align*}
\end{itemize}
    Then $s_{\mathcal{K}}(c)\in \mathcal C^1[c_0-\delta,c_0+\delta]$ and  we have
    \begin{align}
    \label{eq:angles}
       \frac{d}{dc}s_{\mathcal{K}}(c_0)<\frac{d}{dc}\tilde{s}(c_0).
    \end{align}
\end{proposition}
\begin{proof}
    The following relations hold at $u=u_{\mathcal{K}}(u_0)$: 
    \begin{align}
    \label{eq:g1}
        (\lambda^c -\lambda^s)\cdot \frac{ds_{\mathcal{K}}}{dc}=f_c+g,\qquad (\lambda^c-\lambda^s)\cdot \frac{d\tilde{s}}{dc}=f_c,
    \end{align}
    where 
    \begin{align}
        \label{eq:g2}
    g(u)=\lambda^c(u)\cdot a_{cc}(c)\cdot \frac{(s_{\mathcal{K}}(s)-s)}{s+a_{c}(c)}.
    \end{align}
    Note that $g(u_0) < 0$, and the statement of this proposition follows immediately.
\end{proof}

\paragraph{Notation for important points and curves.}In order to describe the structure of solutions to a Riemann problem, we need to introduce some notation. See Figure~\ref{fig:important-points} for an illustration for these new notations.

The four rarefaction curves $s_i(c)$, $i=1,\ldots,4$, which intersect at the fixed point $(s^*,c^*)$ (see Proposition~\ref{prop:rare}), play an important role in the analysis of solutions of Riemann problems. Denote 
\begin{align}
\label{eq:s1L-s2L-s3R-s4R}
    s_{1L}:=s_1(c_L);\qquad s_{2L}:=s_2(c_L);\qquad
    s_{3R}:=s_3(c_R);\qquad
    s_{4R}:=s_4(c_R).
\end{align}
Consider the critical rarefaction values that correspond to the points $s_{1L}, s_{2L}, s_{3R}$:
\begin{align}
\label{eq:critical-s1K}
s_{1K}:&=s_{\mathcal{K}}(s_{1L},c_L);
\qquad     
s_{2K}:=s_{\mathcal{K}}(s_{2L},c_L);
\qquad
s_{3K}:=s_{\mathcal{K}}(s_{3R},c_R).
\end{align}
If $s_{1K}$ does not exist, the analysis of the solutions to the Riemann problem is easier, as there are fewer cases to consider. Thus, we will assume that there exists $s_{1K}<1$. Note also that Lemma~\ref{lm:critical-values-inequalities} provides the existence of $s_{3K}<1$. Consider the $c$-rarefaction curve $\Gamma_0$ defined by $s_0(c)\equiv1$, $c\in[0,1]$ and its critical curve $\Gamma_{0\mathcal{K}}$, parametrized by $s_{0\mathcal{K}}(c)$, $c\in[0,1]$. Denote
\begin{align}
\label{eq:s0L}
    s_{0L}=s_{0\mathcal{K}}(c_L),\qquad s_{0R}=s_{0\mathcal{K}}(c_R).
\end{align}
There exists a unique $s_{0K}$  such that 
\begin{align}
\label{eq:s0K}
    (s_{0K},c_L)\xrightarrow{c\text{-rare}}(s_{0R},c_R).
\end{align}

\begin{figure}[h]
    \centering\includegraphics[width=0.5\linewidth]{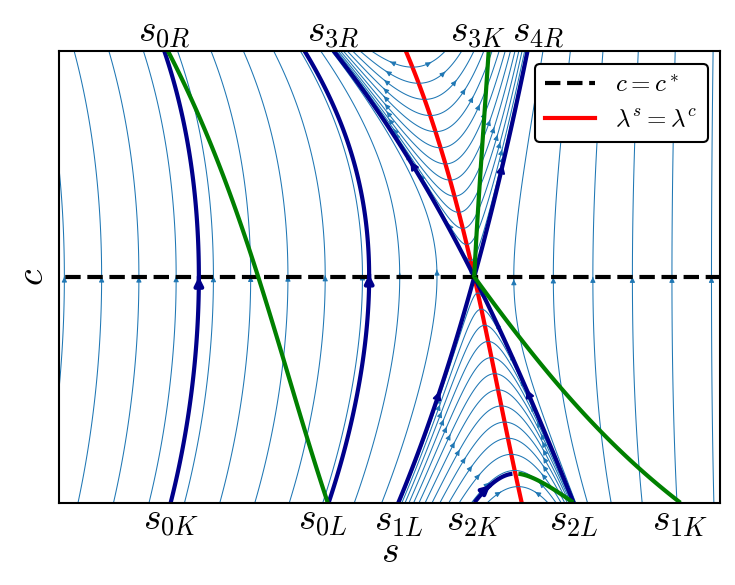}
    \caption{The illustration of the new notations. Note that the points are arranged according to inequalities~\eqref{eq:cL}, \eqref{eq:cR} from Lemma~\ref{lm:critical-values-inequalities}.}
    \label{fig:important-points}
\end{figure}

\begin{lemma} 
\label{lm:critical-values-inequalities}
The following inequalities hold:
\begin{align}
\label{eq:cL}
    s_{0K}<s_{0L}<s_{1L}<s_{2K}<s_{2L}<s_{1K}\quad \text{(abscissas of the states with $c=c_L$)},
    \\
\label{eq:cR}
    s_{0R}<s_{3R}<s_{3K}<s_{4R} < 1\quad \text{(abscissas of the states with $c=c_R$)}.
\end{align}
\end{lemma}
\begin{proof}
    First, let us  prove inequalities~\eqref{eq:cL}. Due to Proposition~\ref{prop:sK-basic-properties}, we get that the assumption $s_{1K}<1$ implies
$s_{0L}<s_{1L}$. In addition, $(s_{2L},c_L)\in\Omega_R$ implies $(s_{2K},c_L)\in\Omega_L$, so $s_{2K}<s_{2L}$. So, it is enough to prove:
\begin{align*}
    s_{0K}<s_{0L}; \qquad s_{1L}<s_{2K};\qquad s_{2L}<s_{1K}.
\end{align*}

Let us prove $s_{1L}<s_{2K}$. The other two inequalities are proved in a similar way. The scheme of the proof is simple: we assume the opposite inequality, and get a contradiction with Proposition~\ref{prop:crit}.
    
Assume first that $s_{1L}>s_{2K}$ (the case of equality is considered below). Consider the rarefaction curve $\Gamma=\{(s(c),c)\}$ that passes through the state $(s_{2K},c_L)$. By Remark~\ref{rm:rare}, $\Gamma\subset\Omega_L$ and is defined for all $c\in[0,1]$.  Consider a curve $\Gamma_{\mathcal{K}}=\{(s_{\mathcal{K}}(c),c)\}\subset\Omega_R$, a critical curve to $\Gamma$. Either $\Gamma_{\mathcal{K}}$ is defined for all $c\in[0,1]$ or there exists some $\tilde{c}$ such that $s_{\mathcal{K}}(\tilde{c})=1$. In both cases, we claim that there exists at least one point of  intersection of the curves $\Gamma_{\mathcal{K}}$ and $\Gamma_2$ (defined in Proposition~\ref{prop:rare}) where the inequality~\eqref{eq:angles} is violated. Indeed, Proposition~\ref{prop:crit} implies that for $c_0$ sufficiently close to $c_L$, we have $s_{\mathcal{K}}(c_0)<s_2(c_0)$. If $\Gamma_{\mathcal{K}}$ is defined for all $c\in[0,1]$, then for $c=c^*$, obviously, $s_{\mathcal{K}}(c^*)>s_2(c^*)=s^*$. As the curves $\Gamma_{\mathcal{K}}$ and $\Gamma_2$ are both continuous, there exists at least one intersection point where $\Gamma_{\mathcal{K}}$ crosses $\Gamma_2$ from left to right, and therefore the inequality~\eqref{eq:angles} is violated at this point. For the second case, $\Gamma_{\mathcal{K}}$ connects two states $(s_{\mathcal{K}}(c_0),c_0)$ and $(1,\tilde{c})$ on different sides of $\Gamma_2$, thus, by continuity,  $\Gamma_{\mathcal{K}}$ and $\Gamma_2$ similarly intersect.

Assume now the equality case $s_{1L}=s_{2K}$. Consider $\Gamma_{\mathcal{K}}=\{(s_{\mathcal{K}}(c),c)\}\subset\Omega_L$, the critical curve to the rarefaction curve~$\Gamma_2=\{(s_{2}(c),c)\}\subset\Omega_R$. 
Note that \mbox{$(s_2(c^*),c^*)=(s^*,c^*)\in\mathcal{C}$}, thus \mbox{$(s_{\mathcal{K}}(c^*),c^*)=(s^*,c^*)\in\mathcal{C}$}; in particular \mbox{$s_{\mathcal{K}}(c^*)=s^*$}. Moreover, similarly to the previous case, Proposition~\ref{prop:crit} implies that there exists $c_0>c_L$, sufficiently close to $c_L$, such that $s_{\mathcal{K}}(c_0)<s_1(c_0)$. Take any point between $(s_{\mathcal{K}}(c_0), c_0)$ and $(s_1(c_0), c_0)$, for example $(s_0, c_0)$ with $s_0:=\frac{1}{2}(s_{\mathcal{K}}(c_0)+s_1(c_0))$, $s_{\mathcal{K}}(c_0) <s_0<s_1(c_0)$, and consider the rarefaction curve~$\Gamma_0=\{(s(c),c)\}$ that passes through the point $(s_0,c_0)$. Due to Remark~\ref{rm:rare}, $\Gamma_0\in\Omega_L$ and intersects $c=c^*$ at some point $(\tilde{s},c^*)$ with $\tilde{s}<s^*$. Thus, $\Gamma_{\mathcal{K}}$ and $\Gamma_0$ intersect and, by construction, inequality~\eqref{eq:angles} is violated at the point of intersection.   

Inequalities~\eqref{eq:cR} are proved in a similar way. We omit the proof.
\end{proof}

\paragraph{Riemann problem solutions structures: case of zero adsorption. }

\mbox{Although} it is not the focus of this paper, let us first discuss the structure of solutions to the Riemann problem for $a(c)\equiv0$ (zero adsorption). It will ease us into the more complicated case and allow us to compare the layouts qualitatively. For the monotone case with zero adsorption, the solutions of the Riemann problem originally were constructed in~\cite{Eli1981} using the generalized Lax entropy condition (see also~\cite{KK} and later works \cite{Eli-Temple},~\cite{Souza-Marchesin}). In~\cite{PePlMa2024} the vanishing adsorption limit of the Riemann problem solution was analyzed. The selection principle, introduced in~\cite{PePlMa2024}, comes from physical considerations,  justifies the admissibility criteria adopted previously for the monotone case, and selects the undercompressive contact discontinuities required to solve the general Riemann problem with non-monotone dependence (see~\cite[formula~(4.9)]{PePlMa2024}, which is an analogue of formula~\eqref{eq:solution-RP-1} for the case of zero adsorption).

Note that for the case of zero adsorption, Proposition~\ref{prop:crit} is not valid ($g\equiv0$ due to formula~\eqref{eq:g2}). Thus, formula~\eqref{eq:g1} implies that if the critical curve $\Gamma_{\mathcal{K}}$ intersects some $c$-rarefaction curve $\widetilde{\Gamma}$, then they coincide. In fact, these $c$-rarefaction curves are contact discontinuities. Also, 
\begin{align*}
    s_{0K}=s_{0L};\quad s_{1L}=s_{2K};\quad
    s_{1K}=s_{2L};\quad s_{3R}=s_{4K};\quad s_{4R}=s_{3K}.\quad
\end{align*}
This simplifies the description of all possible structures of the solution to the Riemann problem as a sequence of $s$-wave groups and $c$-waves. Depending on  $(s_L,s_R)\in[0,1]^2$, we obtain a subdivision into three regions, see Fig.~\ref{fig:RP-layouts-rare}a:
\begin{enumerate}
    \item If $(s_L,s_R)\in\mathbf{U}_{cs}$, then 
    \begin{align}
\label{eq:solution-RP-cs-a=0}
    u_L=(s_L,c_L)\xrightarrow{c\text{-rare}} u_{M}=(s_M,c_R) \xrightarrow{s} u_R=(s_R,c_R).
\end{align}
Here $s_L\in[0,s_{1L}]$ and $s_{R}\in[0,\min(1, s_{\mathcal{K}}(u_M))]$.
\item If $(s_L,s_R)\in\mathbf{U}_{sc}$, then 
\begin{align}
\label{eq:solution-RP-sc-a=0}
    u_L=(s_L,c_L)\xrightarrow{s}
    u_{N}=(s_N,c_L) \xrightarrow{c\text{-rare}}  
    u_R=(s_R,c_R).
\end{align}
Here $s_R\in[s_{4R},1]$ and $s_{L}\in[s_{\mathcal{K}}(u_N),1]$.
\item If $(s_L,s_R)\in\mathbf{U}_{scs}$, then 
    \begin{align}
\label{eq:solution-RP-scs-a=0}
    u_L=(s_L,c_L)\xrightarrow{s} u_{-}=(s_{2L},c_L)\xrightarrow{c\text{-rare}} u_{+}=(s_{3R},c_R) \xrightarrow{s} u_R=(s_R,c_R).
    \end{align}
    Here $s_R\in[0,s_{4R}]$ and $s_{L}\in[s_{1L},1]$. 
\end{enumerate}
Notice that the combinations of waves~\eqref{eq:solution-RP-cs-a=0}, \eqref{eq:solution-RP-sc-a=0}, \eqref{eq:solution-RP-scs-a=0} are compatible due to Proposition~\ref{prop:critical-value-properties}, which is also valid for the case of zero adsorption. 
\begin{remark}
Notice that on the boundaries of the regions~$\mathbf{U}_{cs}$, $\mathbf{U}_{sc}$, and $\mathbf{U}_{scs}$, each of the sequences~\eqref{eq:solution-RP-cs-a=0}, \eqref{eq:solution-RP-sc-a=0}, and \eqref{eq:solution-RP-scs-a=0} may become degenerate.
Although these sequences differ in terms of the order of $s$- and $c$-waves, the resulting solution to the Riemann problem $(s(x,t),c(x,t))$ is the same.
    
\end{remark}

\begin{figure}[H]
    \centering
(a)\includegraphics[width=0.37\linewidth]{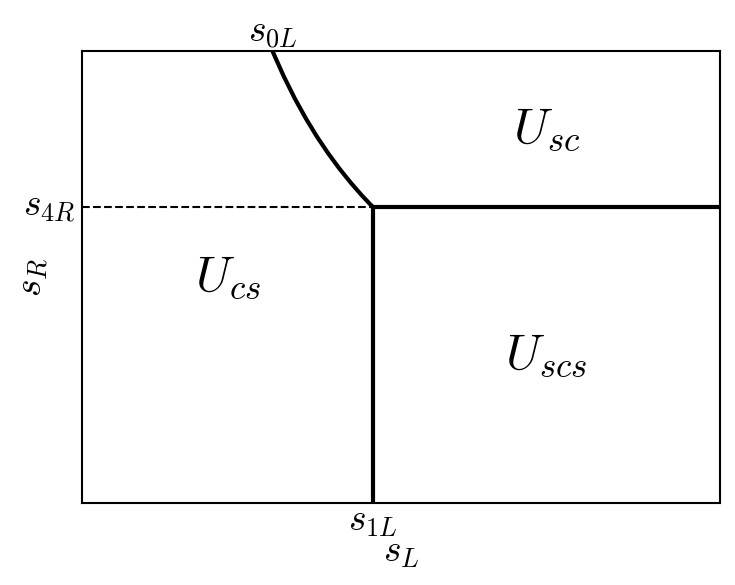}
(b)\includegraphics[width=0.45\linewidth]{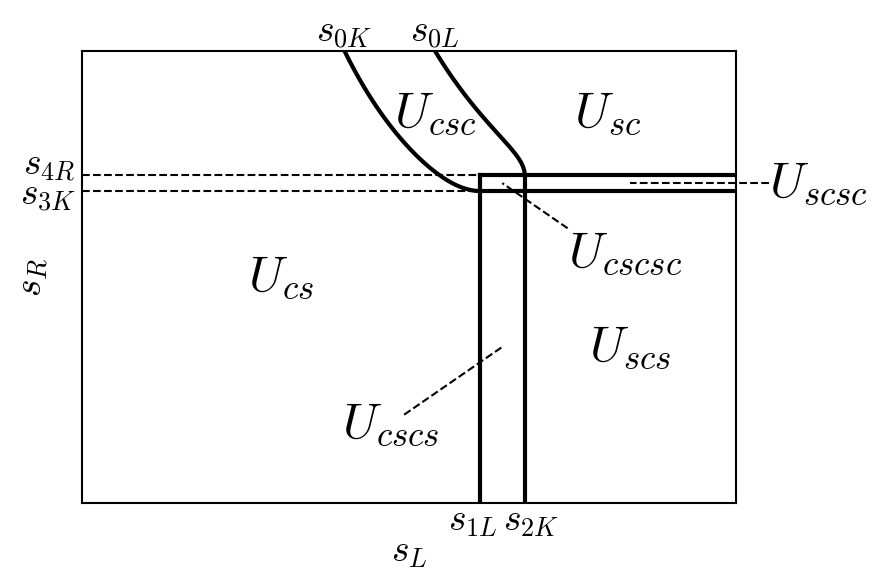}
    \caption{
    Subdivision into regions in the $(s_L,s_R)$-plane with different structure of solutions in terms of the sequence of $s$-wave groups and $c$-waves:\\
    (a) Case with zero adsorption --- three regions $\mathbf{U}_{cs}, \mathbf{U}_{sc}$ and $\mathbf{U}_{scs}$.
    \\
    (b) Case with non-zero adsorption --- 
     seven regions $\mathbf{U}_{cs}, \mathbf{U}_{sc}$, $\mathbf{U}_{scs}$, $\mathbf{U}_{csc}$, $\mathbf{U}_{scsc}$, $\mathbf{U}_{cscs}$ and $\mathbf{U}_{cscsc}$. In the limit of vanishing adsorption, $a(c)\to0$, the regions $\mathbf{U}_{csc}$, $\mathbf{U}_{scsc}$, $\mathbf{U}_{cscs}$ and $\mathbf{U}_{cscsc}$ dissapear and the diagram (b) tends to diagram (a).}
    \label{fig:RP-layouts-rare}
\end{figure}

\begin{figure}[H]
    \centering
\includegraphics[width=0.325\linewidth]{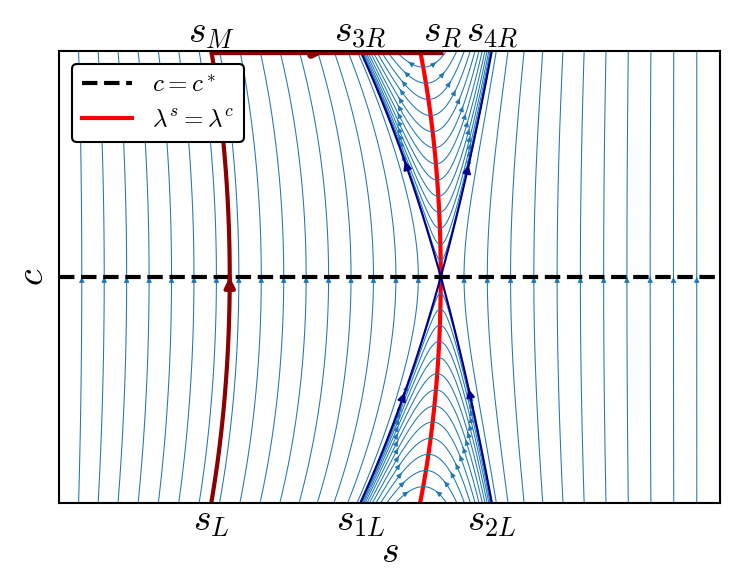}
    \includegraphics[width=0.325\linewidth]{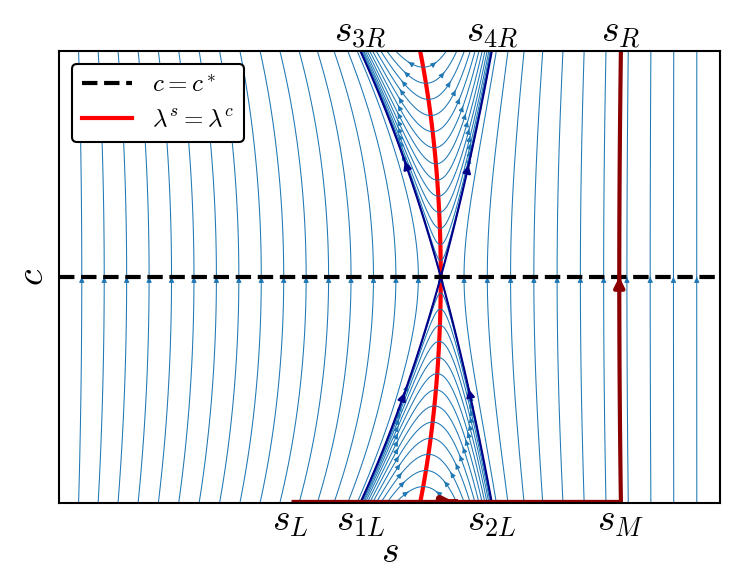}
    \includegraphics[width=0.325\linewidth]{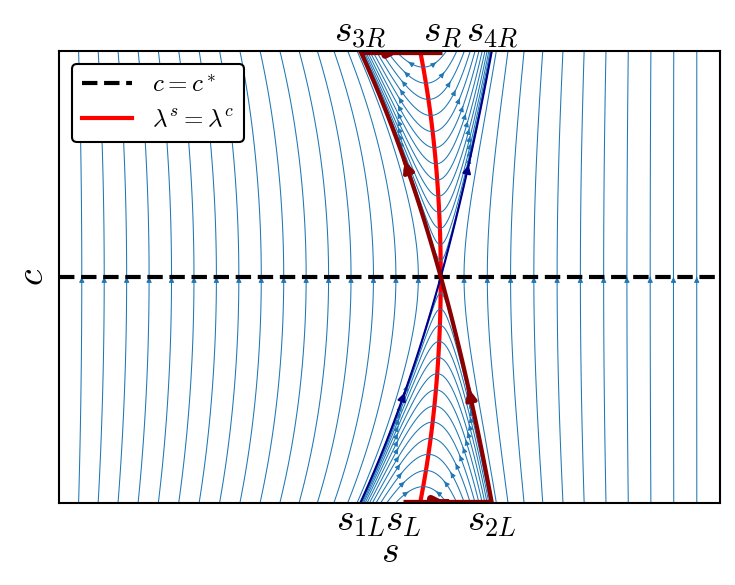}
    \\  \qquad(a)\qquad\qquad\qquad\qquad\qquad\qquad (b)\qquad\qquad\qquad\qquad\qquad\qquad (c)
    \caption{
    (a) Representation of the sequence of waves~\eqref{eq:solution-RP-cs-a=0} in $(s,c)$-plane for $(s_L,s_R)\in\mathbf{U}_{cs}$.
    \\
    (b) \,Representation of the sequence of waves~\eqref{eq:solution-RP-sc-a=0} in $(s,c)$-plane for $(s_L,s_R)\in\mathbf{U}_{sc}$.
 \\
    (c) Representation of the sequence of waves~\eqref{eq:solution-RP-scs-a=0} in $(s,c)$-plane for $(s_L,s_R)\in\mathbf{U}_{scs}$.}
    \label{fig:RP-regions-cs-sc-scs}
\end{figure}

\paragraph{Riemann problem solutions structures: non-zero adsorption case. }

In this Section we list all possible structures of the solution to the Riemann problem for $a(c)$ satisfying (A1)--(A3) as a sequence of $s$-wave groups and $c$-waves depending on  $(s_L,s_R)\in[0,1]^2$. We obtain the subdivision into 7 regions (see Fig.~\ref{fig:RP-layouts-rare}b):

\begin{enumerate}
    \item Region $\mathbf{U}_{cs}$:
    \begin{align*}
        \mathbf{U}_{cs}=\{(s_L,s_R)\in[0,1]^2:
        \,&s_L\in[0,s_{1L}]\text{ and }s_R\in[0,s_{\mathcal{K}}(u_M)] \text{ when } s_{\mathcal{K}}(u_M) \text{ exists},
        \\
        &\text{or }s_L\in[0,s_{1L}] \text{ and }s_R\in[0,1] \text{ when it does not exist;}
        \\
        &\text{here } u_M \text{ is such that } u_L=(s_L,c_L)\xrightarrow{c\text{-rare}}u_M=(s_M,c_R)\}.
    \end{align*}
    Notice that by Remarks~\ref{rmk:rare-s=0-s=1},~\ref{rm:rare} for $s_L\in[0,s_{1L})$ the rarefaction curve that passes through the state $(s_L,c_L)$ always reaches some state on the line $\{c=c_R\}$, which we call $u_M=(s_M,c_R)$. For $s_L=s_{1L}$, we take $s_M=s_{3R}$ and the combination of two $c$-rarefaction curves $\Gamma_1$ and $\Gamma_3$ connects $u_L$ and $u_M$. Due to Proposition~\ref{prop:critical-value-properties}, for any $(s_L,s_R)\in\mathbf{U}_{cs}$ the following sequence of waves $(cs)$ provides a solution to the Riemann problem:
    \begin{align}
\label{eq:solution-RP-cs}
    u_L=(s_L,c_L)\xrightarrow{c\text{-rare}} u_{M}=(s_M,c_R) \xrightarrow{s} u_R=(s_R,c_R).
\end{align}
Geometrically, on the plane $(s_L,s_R)$  the region $\mathbf{U}_{cs}$ is a union of rectangle $[0,s_{0K}]\times[0,1]$ and a curvilinear trapezoid with three straight boundaries: 
\begin{align*}
    \{s_{0K}\}\times[0,1];\qquad  [s_{0K},s_{1L}]\times\{0\};\qquad \{s_{1L}\}\times [0,s_{3K}];
\end{align*}
and the forth boundary being a continuous (strictly decreasing) curve
\begin{align*}
 b_{cs}:s_L\in[s_{0K},s_{1L}]\mapsto s_R\in [s_{3K},1], &\text{ where } s_R=s_{\mathcal{K}}(u_M),
\end{align*}
for $u_M$ defined in~\eqref{eq:solution-RP-cs}. In particular, the inverse $b_{cs}^{-1}: s_R\mapsto s_L$ is well-defined.

The sequence of waves~\eqref{eq:solution-RP-cs} degenerates into a single $c$-wave for $s_R=s_M$, specifically for $(s_L,s_R)$ equal to $(0,0)$ and $(s_{1L},s_{3R})$ and a certain continuous monotone curve connecting them.

\item Region $\mathbf{U}_{sc}$:
  \begin{align*}
        \mathbf{U}_{sc}=\{(s_L,s_R)\in[0,1]^2:
        \,&s_R\in[s_{4R},1]\text{ and } s_L\in[s_{\mathcal{K}}(u_N),1];
        \\
        &\text{here } u_N=(s_N,c_L)\xrightarrow{c\text{-rare}}u_R=(s_R,c_R)\}.
    \end{align*}
     Notice that by Remark~\ref{rm:rare} for $s_R\in(s_{4R},1]$ the $c$-rarefaction curve that passes through the state $(s_R,c_R)$ always begins at some state on the line $\{c=c_L\}$, which we call $u_N=(s_N,c_L)$. For $s_R=s_{4R}$, we take $s_N=s_{2L}$ and the combination of two $c$-rarefaction curves $\Gamma_2$ and $\Gamma_4$ connects $u_N$ and $u_R$. Due to Proposition~\ref{prop:critical-value-properties}, for any $(s_L,s_R)\in\mathbf{U}_{sc}$ the following sequence of waves $(sc)$ provides a solution to the Riemann problem:
    \begin{align}
\label{eq:solution-RP-sc}
    u_L=(s_L,c_L)\xrightarrow{s}
    u_{N}=(s_N,c_L) \xrightarrow{c\text{-rare}}  
    u_R=(s_R,c_R).
\end{align}
Geometrically, on the plane $(s_L,s_R)$  the region $\mathbf{U}_{sc}$ is a curvilinear trapezoid with three straight boundaries: 
\begin{align*}
      [s_{2K},1]\times\{s_{4R}\};
      \qquad \{1\}\times[s_{4R},1];
      \qquad[s_{0L},1]\times\{1\};
\end{align*}
and the forth boundary being a continuous (strictly decreasing) curve
\begin{align*}
    b_{sc}:s_R\in[s_{4R},1]\mapsto s_L\in [s_{0L},s_{2K}], &\text{ where } s_L=s_{\mathcal{K}}(u_N),
\end{align*}
for $u_N$ defined in~\eqref{eq:solution-RP-sc}. In particular, the inverse $b_{sc}^{-1}: s_L\mapsto s_R$ is well-defined.

The sequence of waves~\eqref{eq:solution-RP-sc} degenerates into a single $c$-wave for $s_L=s_N$, specifically for $(s_L,s_R)$ equal to $(1,1)$ and $(s_{2L},s_{4R})$ and a certain continuous monotone curve between them.

    \item Region $\mathbf{U}_{csc}$:
    \begin{align*}
        \mathbf{U}_{csc}=\{(s_L,s_R)\in[0,1]^2:
        &\text{ either } s_R\in[s_{4R},1] \text{ and }s_L\in[b_{cs}^{-1}(s_R), b_{sc}(s_R)];
        \\
        &\text{ or }
    s_R\in(s_{3K},s_{4R}] \text{ and }s_L\in[b_{cs}^{-1}(s_R), s_{1L}]\}.
    \end{align*}

     We claim that there exist states $u_-,u_+\in[0,1]^2$ such that the following sequence of waves $(csc)$ provides a solution to the Riemann problem:
    \begin{align}
\label{eq:solution-RP-csc}
    u_L=(s_L,c_L)\xrightarrow{c\text{-rare}} u_{-} \xrightarrow{s} u_{+} \xrightarrow{c\text{-rare}} u_R=(s_R,c_R).
    \end{align}
    \begin{remark}
        
    Notice that when $(s_L,s_R)\in\mathbf{U}_{csc}\cap\mathbf{U}_{sc}$ the sequence~\eqref{eq:solution-RP-csc} degenerates into a sequence $(sc)$ of two waves as in~\eqref{eq:solution-RP-sc}. Similarly, when $(s_L,s_R)\in\mathbf{U}_{csc}\cap\mathbf{U}_{cs}$ the sequence~\eqref{eq:solution-RP-csc} degenerates into a sequence $(cs)$ of two waves as in~\eqref{eq:solution-RP-cs}. Also, for $(s_L,s_R)=(s_{1L},s_{4R})$ the solution to a Riemann problem consists of a single $c$-rarefaction wave that corresponds to two $c$-rarefaction curves $\Gamma_1\cup\Gamma_4$. In what follows, we exclude these cases from consideration.
\end{remark}

    The general idea to find the states  $u_-$ and $u_+$ is as follows. Consider $c$-rarefaction curve  $\Gamma_{\mathcal{L}}$ that passes through the state $(s_L,c_L)$, its critical curve $\Gamma_{\mathcal{K}}$ and $c$-rarefaction curve $\Gamma_{\mathcal{R}}$ that passes through the state $(s_R,c_R)$. It is sufficient to prove that there exists an intersection point $u_+=(s_+,c_+)$ of $\Gamma_{\mathcal{R}}$ and $\Gamma_{\mathcal{K}}$. Then taking $u_-=(s_{\mathcal{K}}(u_+),c_+)$, the sequence of waves~\eqref{eq:solution-RP-csc} is admissible due to Proposition~\ref{prop:critical-value-properties}.  Let us prove that $\Gamma_{\mathcal{K}}$ and $\Gamma_{\mathcal{R}}$ intersect for the following cases  separately: 
    \begin{enumerate}
    \item $s_L\in(s_{0K},s_{0L})$ and $s_R\in(b_{cs}(s_L), 1]$;
    \item  $s_L\in[s_{0L},s_{1L}]$ $\,$ and $s_R\in[s_{4R}, b_{sc}^{-1}(s_L))$, $(s_L,s_R)\not=(s_{1L},s_{4R})$;
    \item  $s_L\in[s_{0L},s_{1L}]$ $\,$ and $s_R\in(b_{cs}(s_L),s_{4R})$;
    \item $s_L\in(s_{1L},s_{2K})$ and $s_R\in[s_{4R},b_{sc}^{-1}(s_L))$.
\end{enumerate}
    
    \vspace{5pt}
    
    Case (b) is the simplest to consider (see Fig.~\ref{fig:RP-layouts-rare-1}b). There exist two $c$-rarefaction curves  $\Gamma_{\mathcal{L}}$ and $\Gamma_{\mathcal{R}}$:
    \begin{align*}
        u_L=(s_L,c_L)\xrightarrow{\Gamma_{\mathcal{L}}} u_M=(s_M,c_R);\qquad u_N=(s_N,c_L)\xrightarrow{\Gamma_{\mathcal{R}}} u_R=(s_R,c_R).
    \end{align*}
    If $s_L=s_{1L}$, we take $\Gamma_{\mathcal{L}}=\Gamma_1\cup\Gamma_3$; if $s_R=s_{4R}$, we take $\Gamma_{\mathcal{R}}=\Gamma_2\cup\Gamma_4$.
    For $(s_L,s_R)\in\mathrm{int} (\mathbf{U}_{csc})$, we obtain $s_R>s_{\mathcal{K}}(u_M)$ and $s_L<s_{\mathcal{K}}(u_N)$. Due to Proposition~\ref{prop:sK-basic-properties}, the last inequality implies $s_{\mathcal{K}}(u_L)>s_N$ as long as $s_{\mathcal{K}}(u_L)$ is well-defined (that is for $s_L\geq s_{0L}$). Note that  $\Gamma_{\mathcal{K}}$ is a continuous curve that connects the states
    \begin{align}
    \label{eq:gamma-k}
        (s_{\mathcal{K}}(u_L),c_L)\xrightarrow{\Gamma_{\mathcal{K}}}(s_{\mathcal{K}}(u_M),c_R).
    \end{align}
    Thus, the inequalities $s_{\mathcal{K}}(u_L)>s_N$ and $s_{\mathcal{K}}(u_M)<s_R$ imply that the curves $\Gamma_{\mathcal{R}}$ and $\Gamma_{\mathcal{K}}$ have at least one intersection point (in fact, exactly one due to Proposition~\ref{prop:crit}). 
    
    \medskip
Case (a), see Fig.~\ref{fig:RP-layouts-rare-1}a. The situation here is more involved. By construction, the $c$-rarefaction curve $\Gamma_{\mathcal{L}}$ has exactly one intersection point with~$\Gamma_{0K}$ (critical curve to $c$-rarefaction curve $s\equiv1$). Denote this point by  $(\tilde{s},\tilde{c})$. Thus, the critical curve $\Gamma_{\mathcal{K}}$  connects the states $(1,\tilde{c})$ and $(s_{\mathcal{K}}(u_M),c_R)$. It is clear that $c$-rarefaction curve $\Gamma_{\mathcal{R}}$ intersects the line $\{c=\tilde{c}\}$ at some state where $s<1$, hence, by continuity the curves $\Gamma_{\mathcal{R}}$ and $\Gamma_{\mathcal{K}}$ intersect at some point.

   \vspace{5pt}
   Case (c), see Fig.~\ref{fig:RP-layouts-rare-1}c. 
   The main difference with the case (b) is that $\Gamma_{\mathcal{R}}$ crosses the coincidence locus $\mathcal{C}$ at some point $(\tilde{s},\tilde{c})$, does not cross the line $\{c=c^*\}$. Notice that the critical curve  $\Gamma_{\mathcal{K}}\subset\Omega_R$, thus due to~\eqref{eq:gamma-k} it intersects the line $\{c=\tilde{c}\}$ at some point in $\Omega_R$, where $s>\tilde{s}$. The same argument as before implies that the curves $\Gamma_{\mathcal{K}}$ and $\Gamma_{\mathcal{R}}$ intersect.

\vspace{5pt}
    Case (d), see Fig.~\ref{fig:RP-layouts-rare-1}d.  
    The main difference with the case (b) is that $\Gamma_{\mathcal{L}}$ crosses the coincidence locus $\mathcal{C}$ at some point $(\tilde{s},\tilde{c})$ and does not cross the line $\{c=c^*\}$.     For $(s_L,s_R)\in\mathrm{int}(\mathbf{U}_{csc})$, we  obtain $s_L<s_{\mathcal{K}}(u_N)$ or equivalently $s_{\mathcal{K}}(u_L)>s_N$. Also $\Gamma_{\mathcal{K}}$ intersects the coincidence locus at the same point $(\tilde{s},\tilde{c})$ as the curve $\Gamma_{\mathcal{L}}$. The curve $\Gamma_{\mathcal{R}}$ intersects the segment $\{c=\tilde{c}\}$ at some point in $\Omega_R$, where $s>\tilde{s}$. The same argument as before implies that the curves $\Gamma_{\mathcal{K}}$ and $\Gamma_{\mathcal{R}}$ intersect.

    \begin{figure}[h]
    \centering
    (a)
    \includegraphics[width=0.38\linewidth]{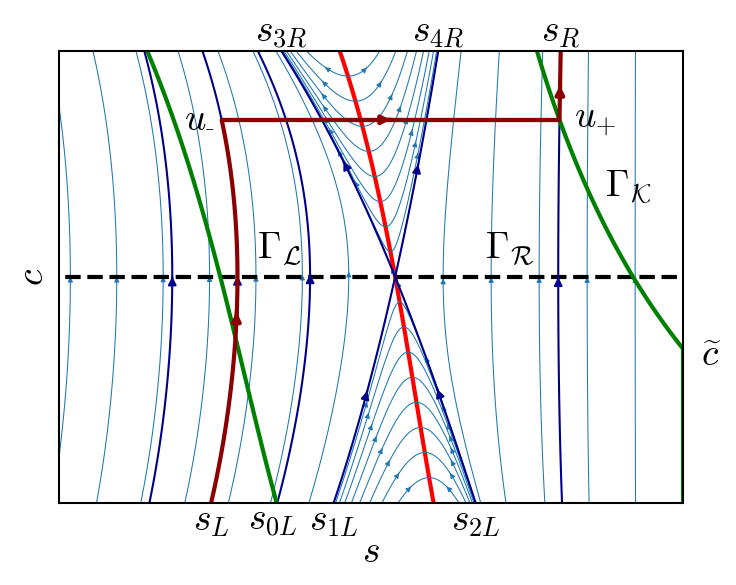}
    (b)\includegraphics[width=0.382\linewidth]{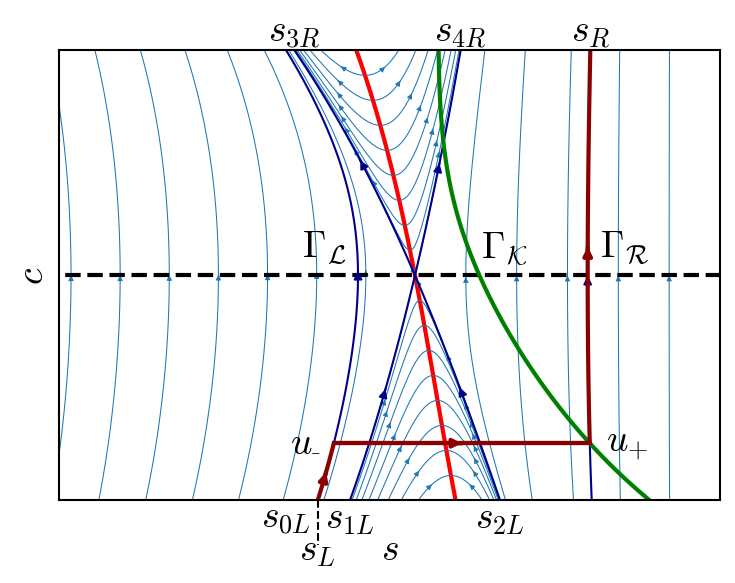}
    \\
    (c)
    \includegraphics[width=0.38\linewidth]{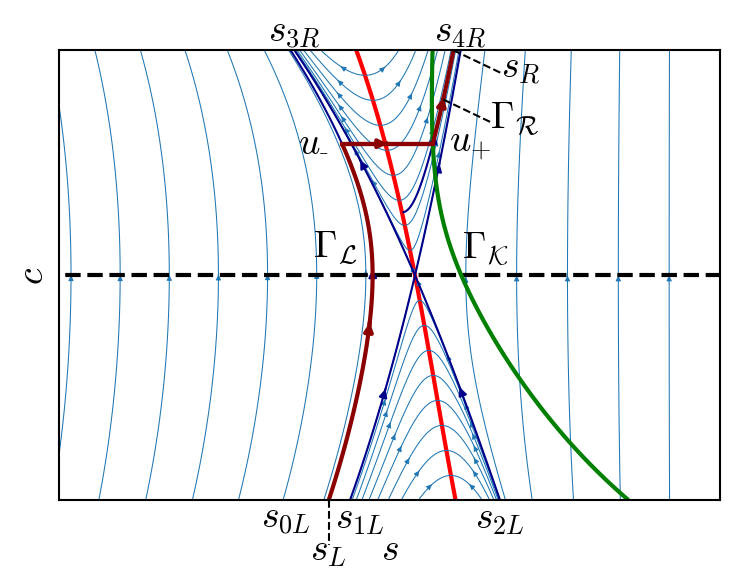}
    (d)\includegraphics[width=0.383\linewidth]{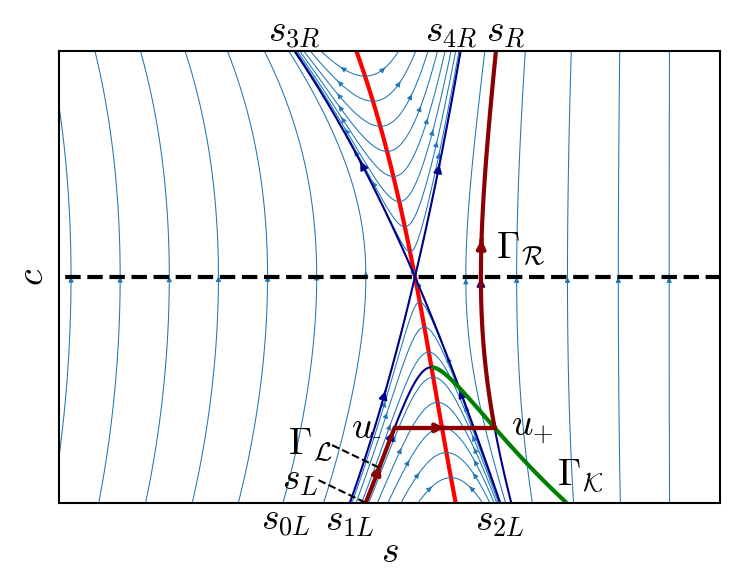}
    \caption{Schematic representation of the sequence of $(csc)$-waves~\eqref{eq:solution-RP-csc} for $(s_L,s_R)\in\mathbf{U}_{csc}$, cases (a), (b), (c) and (d).
    }
    \label{fig:RP-layouts-rare-1}
\end{figure}
  
   \item Region $\mathbf{U}_{scs}$:
    \begin{align*}
        \mathbf{U}_{scs}=\{(s_L,s_R)\in[s_{2K},1]\times[0,s_{3K}]\}.
    \end{align*}
  Due to Proposition~\ref{prop:critical-value-properties}, the following sequence of waves $(scs)$ provides a solution to the Riemann problem:
    \begin{align}
\label{eq:solution-RP-scs}
    u_L=(s_L,c_L)\xrightarrow{s} u_{2L}=(s_{2L},c_L)\xrightarrow{c\text{-rare}} u_{3R}=(s_{3R},c_R) \xrightarrow{s} u_R=(s_R,c_R).
    \end{align}
    Here the $c$-rarefaction curve between $u_{2L}$ and $u_{3R}$ can be seen as a combination of two $c$-rarefaction curves $\Gamma_2$ and $\Gamma_3$:
    \begin{align}
    \label{eq:rare-s2L-s3R}
        u_{2L}=(s_{2L},c_L)\xrightarrow{c\text{-rare}} (s^*,c^*)\xrightarrow{c\text{-rare}}u_{3R}=(s_{3R},c_R).
    \end{align}
    The corresponding  rarefaction waves are given by formulas~\eqref{eq:rare}. Their concatenation gives a unique rarefaction wave that corresponds to the $c$-rarefaction curve between $u_{2L}$ and $u_{3R}$:
    \begin{align*}
        u(x,t) =
\begin{cases}
u_{2L}, & \text{if } x/t < \lambda(u_-), \\[0.3em]
v, & \text{if } x/t = \lambda(v), \\[0.3em]
u_{3R}, & \text{if } x/t > \lambda(u_+).
\end{cases}
    \end{align*}

\begin{remark}
     For $(s_L,s_R)=(s_{2L},s_{3R})\in\mathbf{U}_{scs}$, the sequence~\eqref{eq:solution-RP-scs} degenerates into a single $c$-curve~\eqref{eq:rare-s2L-s3R}. Also, when $s_L=s_{2L}$ and $s_R\in[0,s_{3K}]\setminus\{s_{3R}\}$, the sequence~\eqref{eq:solution-RP-scs} degenerates into a $(cs)$-sequence; when $s_R=s_{3R}$ and $s_L\in[s_{2K},1]\setminus\{s_{2L}\}$, the sequence~\eqref{eq:solution-RP-scs} degenerates into a $(sc)$-sequence.
    \end{remark}
    \item Region $\mathbf{U}_{cscs}$:
     \begin{align*}
        \mathbf{U}_{cscs}=\{(s_L,s_R)\in[s_{1L},s_{2K}]\times[0,s_{3K}]\}.
    \end{align*}
    There exist two states $u_-$ and $u_+$ such that the following sequence of waves $(cscs)$ provides a solution to the Riemann problem (see Fig.~\ref{fig:RP-layouts-rare-5}a):
    \begin{align}
\label{eq:solution-RP-cscs}
    u_L=(s_L,c_L)\xrightarrow{c\text{-rare}} u_{-}\xrightarrow{s}u_{+} \xrightarrow{c\text{-rare}} (s_{3R},c_R)\xrightarrow{s} u_R=(s_R,c_R).
    \end{align}
\begin{remark}
    For some pairs $(s_L,s_R)\in\mathbf{U}_{cscs}$ the sequence~\eqref{eq:solution-RP-cscs} degenerates. Specifically, when $s_L = s_{1L}$, the first $(csc)$ sequence in the structure degenerates into a single $c$-wave; when $s_L = s_{2K}$, the first $c$-wave degenerates and disappears; finally, when $s_R = s_{3R}$, the last $s$-wave degenerates and disappears.
    
\end{remark}

The states  $u_-$ and $u_+$ can be constructed as follows.
Consider $c$-rarefaction curve  $\Gamma_{\mathcal{L}}$ that passes through the state $(s_L,c_L)$ and its critical curve $\Gamma_{\mathcal{K}}$. As $s_L\in[s_{1L},s_{2K}]$, $\Gamma_{\mathcal{L}}$ crosses the coincidence locus $\mathcal{C}$ at some point $(\tilde{s},\tilde{c})$, $\tilde{c}\leq c^*$, and so does $\Gamma_{\mathcal{K}}$. Note that $s_{\mathcal{K}}(u_L)\in[s_{2L},s_{1K}]$. This implies that the curves $\Gamma_{\mathcal{K}}$ and $\Gamma_2$ intersect. We denote this point of intersection by $u_+=(s_+,c_+)$ and take $u_-=(s_{\mathcal{K}}(u_+),c_+)$. By Proposition~\ref{prop:critical-value-properties}, all combinations of $s$ and $c$ waves are compatible by construction.

    \begin{figure}[h]
    \centering
    \includegraphics[width=0.32\linewidth]{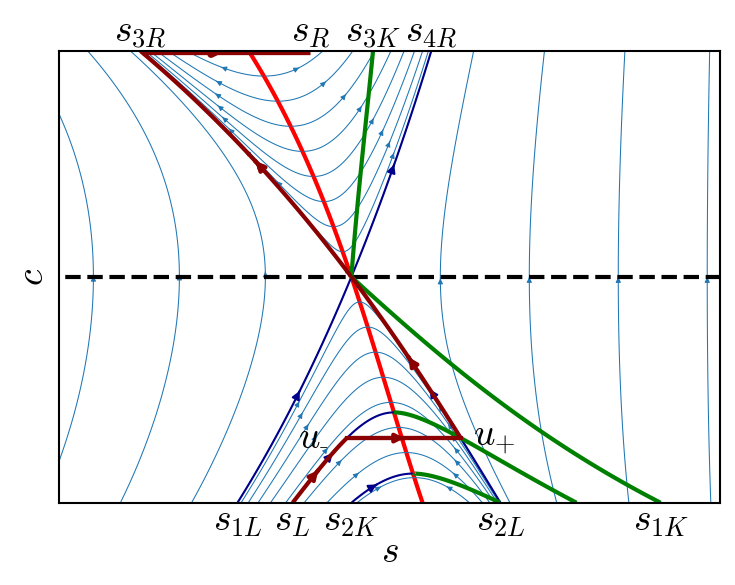}
    \includegraphics[width=0.32\linewidth]{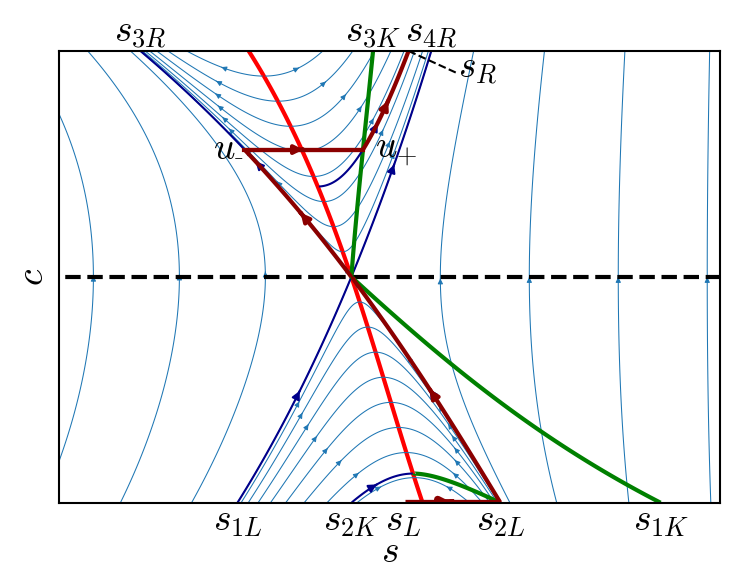}
    \includegraphics[width=0.32\linewidth]{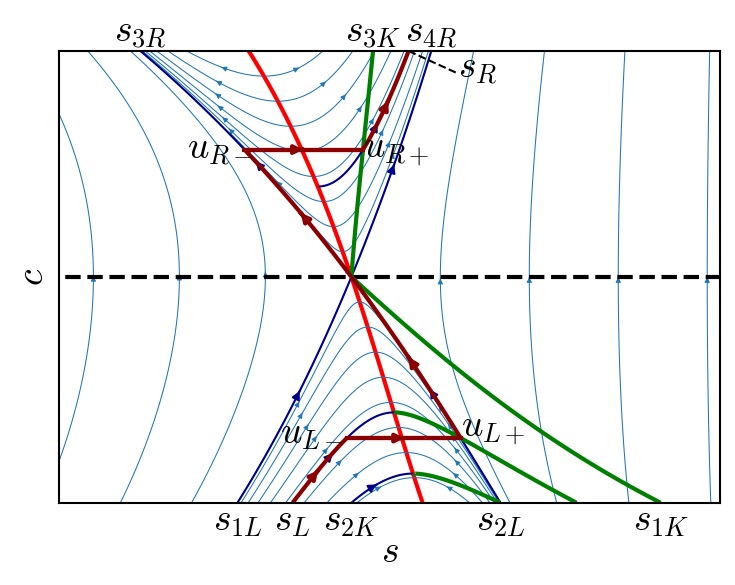}
    \\
    \qquad\qquad\qquad\qquad(a)\hfill (b)\hfill (c)\qquad\qquad\qquad\quad\quad
    \caption{(a) Case $(s_L,s_R)\in \mathbf{U}_{cscs}$; (b) Case $(s_L,s_R)\in \mathbf{U}_{scsc}$; (c) Case $(s_L,s_R)\in \mathbf{U}_{cscsc}$.}
    \label{fig:RP-layouts-rare-5}
\end{figure}

    \item Region $\mathbf{U}_{scsc}$:
    \begin{align*}
        \mathbf{U}_{scsc}=\{(s_L,s_R)\in[s_{2K},1]\times[s_{3K},s_{4R}]\}.
    \end{align*}
There exist two states $u_-$ and $u_+$ such that the following sequence of waves $(scsc)$ provides a solution to the Riemann problem (see Fig.~\ref{fig:RP-layouts-rare-5}b):
    \begin{align}
\label{eq:solution-RP-scsc}
    u_L=(s_L,c_L)\xrightarrow{s}(s_{2L},c_L)\xrightarrow{c\text{-rare}} u_{-}\xrightarrow{s}u_{+} \xrightarrow{c\text{-rare}}  u_R=(s_R,c_R).
    \end{align}
The states  $u_-$ and $u_+$ can be constructed as follows.
Consider $c$-rarefaction curve  $\Gamma_{\mathcal{R}}$ that passes through the state $(s_R,c_R)$ and the critical curve to $c$-rarefaction curve $\Gamma_3$ (we denote it by  $\Gamma_{3\mathcal{K}}$). As $s_R\in[s_{3K},s_{4R}]$, $\Gamma_{\mathcal{R}}$ crosses the coincidence locus $\mathcal{C}$ at some point $(\tilde{s},\tilde{c})$, $\tilde{c}\geq c^*$. This implies that the curves $\Gamma_{\mathcal{R}}$ and $\Gamma_{3\mathcal{K}}$ intersect. We denote this point of intersection by $u_+=(s_+,c_+)$ and take $u_-=(s_{\mathcal{K}}(u_+),c_+)$. By Proposition~\ref{prop:critical-value-properties}, all combinations of $s$ and $c$ waves are compatible by construction. Similar to the previous case, some waves degenerate when $s_L = s_{2L}$, or $s_R=s_{4R}$, or $s_R=s_{3K}$.

    \item Region $\mathbf{U}_{cscsc}$:
    \begin{align*}
        \mathbf{U}_{cscsc}=\{(s_L,s_R)\in[s_{1L},s_{2K}]\times[s_{3K},s_{4R}]\}.
    \end{align*}
There exist four states $u_{L-}$, $u_{L+}$, $u_{R-}$ and $u_{R+}$ such that the following sequence of waves $(cscsc)$ provides a solution to the Riemann problem:
    \begin{align}
\label{eq:solution-RP-cscsc}
    u_L\xrightarrow{c\text{-rare}} u_{L-}\xrightarrow{s}u_{L+} \xrightarrow{c\text{-rare}} u_{R-}\xrightarrow{s} u_{R+} \xrightarrow{c\text{-rare}} u_R.
    \end{align}
The states  $u_{L-}$ and $u_{L+}$ can be constructed exactly in the same way as the states $u_-$ and $u_+$ for the case $\mathbf{U}_{cscs}$. The states  $u_{R-}$ and $u_{R+}$ can be constructed exactly in the same way as the states $u_-$ and $u_+$ for the case $\mathbf{U}_{scsc}$. See Fig.~\ref{fig:RP-layouts-rare-5}c. 
Similar to previous cases, some waves degenerate at $s_L = s_{1L}$ or $s_L = s_{2K}$ and $s_R = s_{4R}$ or $s_R = s_{3K}$.
Note that~\eqref{eq:solution-RP-cscsc} is the exact sequence of waves that contains three different $c$-rarefaction curves and was not described in previous works on this problem.
\end{enumerate}

\section*{Acknowledgements}

The authors thank Pavel Bedrikovetsky for lectures on systems of hyperbolic conservation laws. The work of Y. Petrova was supported by Programa de Apoio a Novos Docentes USP, CNPq grant 406460/2023-0. The work of N. Rastegaev was supported by the Ministry of Science and Higher Education of the Russian Federation (agreement 075-15-2025-344 dated 29/04/2025 for Saint Petersburg Leonhard Euler International Mathematical Institute at PDMI RAS).
\bigskip
\bigskip

\begin{enbibliography}{99}
\addcontentsline{toc}{section}{References}

\bibitem{Pires2021}
Apolin\'{a}rio, F.~O. and Pires, A.~P., 2021. Oil displacement by multicomponent slug injection: An analytical solution for Langmuir adsorption isotherm. Journal of Petroleum Science and Engineering, 197, p.~107939.

\bibitem{Bahetal}
Bakharev, F., Enin, A., Petrova, Y. and Rastegaev, N., 2023. Impact of dissipation ratio on vanishing viscosity solutions of the Riemann problem for chemical flooding model. Journal of Hyperbolic Differential Equations, 20(02), pp.~407-432.

\bibitem{Tapering}
Bakharev, F., Enin, A., Kalinin, K., Petrova, Y., Rastegaev, N. and Tikhomirov, S., 2023. Optimal polymer slugs injection profiles. Journal of Computational and Applied Mathematics, 425, p.~115042.

\bibitem{BL} 
Buckley, S.~E. and Leverett, M., 1942. Mechanism of fluid displacement in sands. Transactions of the AIME, 146(01), pp.~107-116.

\bibitem{Castaneda}
Casta\~{n}eda, P., 2016. Dogma: S-shaped. The Mathematical Intelligencer, 38, pp.~10-13.

\bibitem{Courant}
Courant, R., 1944. Supersonic Flow and Shock Waves: A Manual on the Mathematical Theory of Non-linear Wave Motion (No. 62). Courant Institute of Mathematical Sciences, New York University.

\bibitem{Dafermos}
Dafermos, C.~M., 2016. Hyperbolic Conservation Laws in Continuum Physics. Springer.

\bibitem{Tveito}
Dahl, O., Johansen, T., Tveito, A. and Winther, R., 1992. Multicomponent chromatography in a two phase environment. SIAM Journal on Applied Mathematics, 52(1), pp.~65-104.

\bibitem{Dan-topological-tool}
Eschenazi, C.~S., Lambert, W.~J., Lopez-Flores, M.~M., Marchesin, D., Palmeira, C.~F. and Plohr, B.~J., 2025. Solving Riemann problems with a topological tool. Journal of Differential Equations, 416, pp.~2134-2174.

\bibitem{Gelfand}
Gelfand, I.~M., 1959. Some problems in the theory of quasilinear equations. Uspekhi Matematicheskikh Nauk, 14(2), pp.~87-158 (in Russian). English translation in Transactions of the American Mathematical Society, 29(2), 1963, pp.~295-381.

\bibitem{Eli1981}
Isaacson, E.~L., 1981. Global solution of a Riemann problem for a non-strictly hyperbolic system of conservation laws arising in enhanced oil recovery, Rockefeller University, New York, NY, preprint.

\bibitem{IsaMarPlohr}
Isaacson, E.~L., Marchesin, D. and Plohr, B.~J., 1990. Transitional waves for conservation laws. SIAM Journal on Mathematical Analysis, 21(4), pp.~837-866.

\bibitem{Eli-Temple}
Isaacson, E.~L., and Temple, J.~B., 1986. Analysis of a singular hyperbolic system of conservation laws, Journal of Differential Equations, 65(2), pp.~250-268.

\bibitem{JnW}
Johansen, T. and Winther, R., 1988. The solution of the Riemann problem for a hyperbolic system of conservation laws modeling polymer flooding. SIAM Journal on Mathematical Analysis, 19(3), pp.~541-566.

\bibitem{KK}
Keyfitz, B.~L., and Kranzer, H.~C., 1980. A system of non-strictly hyperbolic conservation laws arising in elasticity theory. Archive for Rational Mechanics and Analysis, 72(3), pp.~219-241.

\bibitem{Kruzhkov}
Kru\v{z}kov, S.~N., 1970. First order quasilinear equations in several independent variables. Mathematics of the USSR-Sbornik, 10(2), pp.~217-243.

\bibitem{Lax57}
Lax, P.~D., 1957. Hyperbolic systems of conservation laws II. Communications on Pure and Applied Mathematics, 10(4), pp.~537-566.

\bibitem{MR2025}
Matveenko, S. and Rastegaev, N., 2025. Oil displacement by slug injection: a rigorous justification for the Jouguet principle heuristic. arXiv preprint arXiv:2511.08533.

\bibitem{Oleinik}
Oleinik, O.~A., 1957. Discontinuous solutions of non-linear differential equations. Uspekhi Matematicheskikh Nauk, 12(3)(75), pp.~3-73 (in Russian). English translation in American Mathematical Society Translations, 26(2), 1963, pp.~95-172.

\bibitem{PiBeSh06}
Pires, A.~P., Bedrikovetsky, P.~G. and Shapiro, A.~A., 2006. A splitting technique for analytical modelling of two-phase multicomponent flow in porous media. Journal of Petroleum Science and Engineering, 51(1-2), pp.~54-67.

\bibitem{PePlMa2024}
Petrova, Y., Plohr, B.~J. and Marchesin, D., 2024. Vanishing adsorption limit of Riemann problem solutions for the polymer model. Journal of Hyperbolic Differential Equations, 21(02), pp.~299-327.

\bibitem{RastS-Shaped}
Rastegaev, N.~V., 2024. On the sufficient conditions for the S-shaped Buckley--Leverett function. 
Zapiski Nauchnykh Seminarov POMI, 536, pp.247-260.

\bibitem{MR2024}
Rastegaev, N. and Matveenko, S., 2024. Kru\v{z}kov-type uniqueness theorem for the chemical flood conservation law system with local vanishing viscosity admissibility. Journal of Hyperbolic Differential Equations, 21(04), pp.~1003-1043.

\bibitem{Serre1}
Serre, D. Systems of Conservation Laws 1: Hyperbolicity, entropies, shock waves. Cambridge University Press, 1999.

\bibitem{Shen}
Shen, W., 2017. On the uniqueness of vanishing viscosity solutions for Riemann problems for polymer flooding. Nonlinear Differential Equations and Applications NoDEA, 24, pp.~1-25.

\bibitem{Souza-Marchesin}
de Souza, A.~J.,  and Marchesin, D., 1998. Conservation laws possessing contact characteristic fields with singularities. Acta Applicandae Mathematicae, 51(3), pp.~353-364.

\bibitem{Wa87}
Wagner, D.~H., 1987. Equivalence of the Euler and Lagrangian equations of gas dynamics for weak solutions. Journal of Differential Equations, 68(1), pp.~118-136. 

\end{enbibliography}

\end{document}